\documentclass[11pt,authoryear]{article}
\pdfoutput=1
\usepackage{fullpage,amsthm,amsmath,natbib}
\usepackage{graphicx,amsfonts,hyperref,array}
\usepackage[hang,flushmargin]{footmisc}

\newcommand{\bl}[1]{{\mathbf #1}}

\newcommand{\Exp}[1]{{\text{E}}[ \ensuremath{ #1 } ]  }

\newcommand{\Cov}[1]{{\text{Cov}}[ \ensuremath{ #1 } ]  }

\newtheorem{lemma}{Lemma}
\newtheorem{theorem}{Theorem}
\newtheorem{proposition}{Proposition}
\newtheorem{corollary}{Corollary}

\newtheorem{definition}{Definition}
\newcommand{\tr}{{\rm{tr}}}

\DeclareMathOperator*{\argmin}{arg\,min}
\renewcommand{\vec}{{\rm{vec}}}
\begin{document}

\title{Equivariant minimax dominators of the MLE in the array normal model}
\author{David Gerard$^1$  and Peter Hoff$^{1,2}$  \\
Departments of Statistics$^1$ and Biostatistics$^2$ \\
University of Washington}  
\maketitle

 \let\thefootnote\relax\footnotetext{Email: gerard2@uw.edu,  pdhoff@uw.edu. This research was partially supported by NI-CHD grant R01HD067509.  }


\begin{abstract}
Inference about dependencies in a multiway data array can be made using the array normal model, which corresponds to the class of multivariate normal distributions with 
separable covariance matrices. 
Maximum likelihood and Bayesian methods for inference
in the array normal model have 
appeared in the literature, but there have not been 
any results concerning the optimality properties of such estimators. 
In this article, we obtain results for the array normal model that 
are analogous to some classical results concerning covariance estimation 
for the multivariate normal model. We show that 
under a lower triangular product group, a uniformly minimum 
risk equivariant estimator (UMREE) can be obtained via a generalized Bayes 
procedure.  Although this UMREE is minimax and dominates the MLE, 
it can be improved upon via 
an orthogonally equivariant modification. Numerical comparisons of the risks of these estimators show that the 
equivariant estimators can have substantially lower risks than 
 the MLE. 
\end{abstract}

\section{Introduction}
The analysis of array-valued data, or tensor data, is 
of interest to numerous fields, including
psychometrics \citep{kiers2001three}, 
chemometrics \citep{smilde2005multi,bro2006review},
imaging \citep{vasilescu2003multilinear},
signal processing \citep{cichockitensor} and machine learning \citep{tao2005supervised}, among others \citep{kroonenberg2008applied,kolda2009tensor}.
Such data consist of measurements indexed by multiple categorical factors.
For example, 
multivariate measurements on experimental units over time may be represented 
by a three-way array 
$X = \{  x_{i,j,t} \} 
\in \mathbb R^{m\times p\times t}$, with 
$i$ indexing units, $j$ indexing variables and $t$ indexing time. 
Another example is multivariate relational data, where $x_{i,j,k}$ 
is the type-$k$ relationship between person $i$ and person $j$.

Statistical analysis of such data often proceeds by fitting 
a model such as $X= \Theta + E$, where $\Theta$ 
is low-dimensional 
and $E$ represents additive residual variation about $\Theta$. 
Standard models for $\Theta$ include regression models, additive effects models 
(such as those estimated by ANOVA decompositions) and 
unconstrained mean models if replicate observations are available. 
Another popular approach is to model $\Theta$ as being a low-rank array. 
For such models, ordinary least-squares estimates of $\Theta$ can be obtained 
via various types of tensor decompositions, depending on the definition of rank 
being used \citep{de2000best,de2000multilinear,de2008tensor}. 

Less attention has been given to the analysis of the residual variation $E$. 
However, estimating and accounting for such  variation is critical for 
a variety of inferential tasks, such as 
prediction, model-checking, construction 
of confidence intervals, and improved parameter estimation over ordinary 
least squares. One model for variation among the entries of an array 
is the array normal model 
\citep{akdemir2011array,hoff2011separable}
which is an extension of the matrix normal model  \citep{srivastava1979introduction,dawid1981some}, 
often used in the analysis of spatial and temporal data \citep{mardia1993spatial,shitan1995asymptotic,fuentes2006testing}.  
The array normal model is a class of normal distributions that 
are generated by a multilinear operator known as the Tucker product:
A random $K$-way array $X$ taking values 
in $\mathbb R^{p_1 \times \cdots \times p_K}$  has an array normal 
distribution if 
$ X \stackrel{d}{=} \Theta + Z \times\{ A_1,\ldots, A_K \}$,
where ``$\times$'' denotes the Tucker product (described further in Section \ref{section:invariant.measure}), $Z$ is a random array in  $\mathbb R^{p_1 \times \cdots \times p_K}$
having i.i.d.\ standard normal entries, and 
$A_k$ is a $p_k\times p_k$ nonsingular matrix for each $k\in \{1,\ldots, K\}$. 
Letting $\Sigma_k = A_k A_k^T$ and ``$\otimes$'' denote the Kronecker product, we write
\begin{equation}
X \sim N_{p_1\times \cdots \times p_K} ( \Theta ,  
  \Sigma_K \otimes \cdots \otimes \Sigma_1 )  . 
\label{eq:anormmod}
\end{equation}


A maximum likelihood estimate (MLE) for the parameters in 
(\ref{eq:anormmod})
can be obtained via an iterative coordinate descent algorithm 
\citep{hoff2011separable}, which is a generalization of the 
iterative ``flip-flop'' algorithm developed in
\citet{mardia1993spatial} and \citet{dutilleul1999mle}, 
or alternatively the optimization procedures described in 
\citet{wiesel2012convexity}. 
However, 
based on results for the multivariate normal model, one might suspect that 
the MLE lacks desirable optimality properties: In the multivariate 
normal model, \citet{james1961estimation} showed that the MLE 
of the covariance matrix 
is neither admissible nor minimax. This was accomplished by identifying 
a minimax and uniformly optimal equivariant estimator that is different from the 
(equivariant) MLE, and therefore dominates the MLE. 
As pointed out by James and Stein, this equivariant estimator is itself 
inadmissible, and  improvements to this estimator have been developed and studied by 
\cite{stein1975estimation,takemura1983orthogonally,lin1985monte}, and \cite{haff1991variational}, among others. 

This article develops similar results for the array normal model. 
In particular, we obtain a procedure to obtain the uniformly minimum risk 
equivariant estimator (UMREE) under a lower-triangular product group of 
transformations for which the model (\ref{eq:anormmod}) is invariant. 
Unlike for the multivariate normal model, there is 
no simple characterization of this class of  equivariant 
estimators. However, 
results of \cite{zidek1969representation} and \cite{eaton1989group}
can be used to show that the UMREE can be obtained from the 
Bayes decision rule under an improper prior, which we derive in Section \ref{section:invariant.measure}. 
In Section \ref{section:equi.proc} we obtain the posterior distribution under this prior, and 
show how it can be simulated from using a Markov Chain Monte Carlo (MCMC)
algorithm. Specifically, the MCMC algorithm is a Gibbs sampler that involves 
simulation from a class of distributions over covariance matrices, which 
we call the ``mirror-Wishart'' distributions. 

In Section \ref{section:mree} we develop a version of Stein's loss function 
for covariance estimation in the array normal model,
and show how the Gibbs sampler of Section \ref{section:equi.proc} can 
be used to obtain the UMREE for this loss. We discuss an 
orthogonally equivariant improvement to the 
UMREE in Section \ref{section:orthogonal}, which can be seen as analogous to the estimator 
studied by \cite{takemura1983orthogonally}. 
Section \ref{section:sims} compares the risks of the MLE, UMREE and the orthogonally equivariant 
estimator as a function of the dimension of $X$ in a small simulation study. 
A discussion follows in Section \ref{section:discussion}. 
Proofs are contained in an appendix.

\section{An invariant measure for the array normal model}
\label{section:invariant.measure}
\subsection{The array normal model}
\label{section:array.norm}
The array normal model on $\mathbb R^{p_1\times \cdots \times p_K}$
 consists of the distributions of 
random $K$-arrays 
$X \in \mathbb R^{p_1\times \cdots \times p_K}$ 
for which 
\begin{align}
\label{equation:original.param}
 X 
\overset{d}{=} \Theta + Z \times \left\{A_1,\ldots,A_K\right\}
\end{align}
for some $\Theta\in \mathbb R^{p_1\times \cdots \times p_K}$, 
 nonsingular matrices $A_k \in \mathbb R^{p_k\times p_k}, k = 1,\ldots, K$ and 
a random $p_1\times \cdots \times p_K$ array $Z$ with i.i.d.\ standard normal entries. 
Here, ``$\times$'' denotes the \emph{Tucker product}, which 
is defined by the identity
\begin{equation} 
  \vec( Z\times \{ A_1,\ldots, A_K\} )  = (A_K \otimes \cdots \otimes A_1) z,  
\label{eq:tprod}
\end{equation}
where
``$\otimes$'' is the Kronecker product and 
 $z=\vec(Z)$, the vectorization of $Z$. 
This identity can be used to find the covariance of the elements of 
a random array satisfying (\ref{equation:original.param}): 
Letting $x,z,\theta$ be the vectorizations of 
$X,Z,\Theta$, we have 
\begin{align*}
\Cov{x} & = \Exp{ (x-\theta )(x-\theta)^T} \\ 
              & = \Exp{ (A_K\otimes \cdots \otimes A_1) zz^T 
                        (A_K^T\otimes \cdots \otimes A_1^T)   } \\
   &= (A_K\otimes \cdots \otimes A_1) 
      (A_K^T\otimes \cdots \otimes A_1^T)  =
           (A_KA_K^T\otimes \cdots \otimes A_1A_1^T),
\end{align*}
and so the array normal distributions
correspond to the multivariate normal distributions
with separable (Kronecker structured) covariance matrices.

A useful operation related to the Tucker product is the \emph{matricization}
operation, which reshapes an array into a matrix along an index set, or
\emph{mode}. For example, the \emph{mode-$k$ matricization} of
   $Z$ is
the $p_k \times ( \prod_{l:l\neq k} p_l)$-dimensional matrix $Z_{(k)}$
having rows equal to
the vectorizations of the ``slices'' of $Z$ along the $k$th index set.
An important identity involving the Tucker product is that
if $Y = Z\times \{ A_1,\ldots, A_K\}$  then
\begin{equation}
 Y_{(k)}  = A_k Z_{(k)} \left ( A_K^T \otimes \cdots \otimes A_{k+1}^T \otimes A_{k-1}^T \otimes \cdots  \otimes  A_1^T \right ).  
\label{eq:matricize}
\end{equation}
As shown in 
\cite{hoff2011separable},
a direct application of this identity gives
\begin{align*}
E\left[(X_{(k)} - \Theta_{(k)})(X_{(k)} - \Theta_{(k)})^T \right]=c_k A_kA_k^T,
\end{align*}
where $c_k$ is a scalar.
This shows that $A_kA_k^T$ can be interpreted as the covariance 
among the $p_k$ slices of the array $X$
along its $k$th mode.

The array normal model can be parameterized  in terms of a mean array
$\Exp{X}=  \Theta \in \mathbb R^{p_1\times \cdots \times p_K}$ and 
covariance $\Cov{ \vec(X)} = \sigma^2 ( \Sigma_K \otimes \cdots \otimes \Sigma_1) $,  where $\sigma^2>0$ and for each $k$, $\Sigma_k \in \mathcal S_{p_k}^+$, 
the set of $p_k\times p_k$ positive definite matrices.
 To make the parameterization identifiable, we restrict 
the determinant of each $\Sigma_k$ to be one. 
Denote by $\mathcal S_{\bf p}^+$ this parameter space, that is, 
the values of $(\sigma^2,\Sigma_1,\ldots, \Sigma_K)$ for which
$|\Sigma_k|=1$, $k=1,\ldots, K$. 
Under this parameterization, we 
write $X \sim N_{p_1\times \cdots \times p_K}(\Theta, 
    \sigma^2 ( \Sigma_K \otimes \cdots \otimes \Sigma_1)) $ if and only if
$X \overset{d}{=} \Theta + \sigma Z \times \{\Psi_1,\ldots, \Psi_K\},
$
where for each $k$, $\Psi_k$ is a matrix such that 
  $\Psi_k \Psi_k^T = \Sigma_k$.

Given a sample $X_1,\ldots,X_n\sim $ i.i.d.\ $N_{p_1 \times \cdots \times p_K}(\Theta,\sigma^2(\Sigma_K \otimes \cdots \otimes \Sigma_1) )$, the $(K+1)$-array $X$ obtained by ``stacking'' $X_1,\ldots,X_n$ along a $(K+1)$st mode also has an array normal distribution,
\begin{align*}
X \sim N_{p_1 \times \cdots \times p_K \times n}\left(\Theta \circ \mathbf{1}_n,\sigma^2 ( I_n \otimes \Sigma_K \otimes \cdots \otimes \Sigma_1) \right ),
\end{align*}
where $\mathbf{1}_n$ is the $n \times 1$ vector of ones and ``$\circ$'' denotes the outer product. 
If $n>1$ then
covariance estimation for the array normal model
can be reduced to the case that $\Theta = 0$. To see this,
let $H$ be a $(n-1)\times n$ matrix such that 
 $HH^T = I_{n-1}$ and $H  1_{n}  = 0$. This implies that $H^TH = I_n - \mathbf{1}_n\mathbf{1}_n^T/n$.
Letting $Y = X \times \{I_{p_1},\ldots,I_{p_K},H\}$,
 and $Y_{(K+1)}$ be the mode-$(K+1)$ matricization of $Y$, we have
\begin{align*}
\Exp{ Y_{(K+1)} } & =  H \Exp{  X_{(K+1)}  } =  H \bl 1_n \vec( \Theta )^T  = \bl 0, 
\end{align*}
and so $Y$ is mean-zero.
Using identity (\ref{eq:tprod}),
the covariance of $\vec(Y)$ can be shown to be
$\sigma^2( HH^T\otimes \Sigma_K \otimes \cdots \otimes  \Sigma_1) =
 \sigma^2( I_{n-1}\otimes \Sigma_K \otimes \cdots \otimes  \Sigma_1)$,
and so
$Y \sim N_{p_1 \times \cdots \times p_K\times (n-1)}(0,\sigma^2(I_{n-1}\otimes \Sigma_K \otimes \cdots \otimes \Sigma_1) )$.
For the remainder of this paper, we consider covariance estimation in the
case that $\Theta=0$.

\subsection{Model invariance and a right invariant measure}
Consider the model for an i.i.d.\  sample of  size $n$ from a  $p$-variate mean-zero multivariate normal distribution, 
$X \sim N_{p\times n}( 0, I_{n}\otimes \Sigma)$, $\Sigma \in \mathcal S_p^+$. 
Recall that $AX \sim N_{p\times n}(0,I_n \otimes A\Sigma A^T )$  for 
nonsingular matrices $A$, and so  in particular 
this model is invariant under left multiplication of $X$  
by elements of $G_p^+$, the group of lower triangular matrices with positive diagonals. 
An  estimator $\hat \Sigma$ mapping the sample space $\mathbb{R}^{p \times n}$ to $\mathcal{S}_p^+$ is said to be equivariant under this group if $\hat\Sigma(AX) = A\hat \Sigma(X)A^T$ for all $A \in G_{p}^+$ and $X \in \mathbb{R}^{p \times n}$. 
\cite{james1961estimation} characterized the class of equivariant  
estimators for this model, identified 
the UMREE under a particular loss 
function 
and showed that the UMREE is minimax. Additionally, 
as the MLE  $XX^T/n$ is equivariant and different from the UMREE, the 
MLE is dominated by the UMREE.

We pursue analogous results for the array normal model by first 
reparameterizing in terms of the parameter
$\Sigma^{1/2} = ( \sigma, \Psi_1,\ldots, \Psi_K)$, so 
\begin{align}
X \sim N_{p_1 \times \cdots \times p_K \times n}\left( 0, \sigma^2\left (I_n\otimes  \Psi_K\Psi_K^T\otimes \cdots \otimes \Psi_1\Psi_1^T\right) \right), \ 
\label{eq:ltparm}
\end{align}
where $\sigma>0$ and each $\Psi_k$ is in the set 
 $\mathcal{G}_{p_k}^+$ 
of $p_k\times p_k$ lower triangular matrices with positive diagonals 
 and determinant 1. In this parameterization, $\Psi_k$ is the lower triangular Cholesky 
square root of the mode-$k$ covariance matrix $\Sigma_k$ described 
in Section \ref{section:array.norm}. 

Define the group $\mathcal G_{\mathbf p}^+$ as 
\begin{align*}
  \mathcal{G}_{\mathbf{p}}^+ = \left\{A= (a,A_1,\ldots,A_K) \hspace{1 mm} : \hspace{1 mm} a > 0, A_k \in \mathcal{G}_{p_k}^+ \mbox{ for } k = 1,\ldots,K\right\},
\end{align*}
where the group operation is 
\begin{align*}
AT =   (a,A_1,\ldots,A_K)(t,T_1,\ldots,T_K) = (at,A_1T_1,\ldots,A_KT_K).
\end{align*}
Note that 
 $\mathcal G_{\mathbf p}^+$ 
consists of the same set 
as the parameter 
space for the model, as parameterized in (\ref{eq:ltparm}).
If the group $\mathcal{G}_{\mathbf{p}}^+$ acts on the sample space by
\[   g:X \mapsto a X \times \{A_1,\ldots,A_K,I_n\},  \]
then as shown in  \citet{hoff2011separable} it acts on the parameter space by
\[g:(\sigma,\Psi_1,\ldots,\Psi_K) \mapsto (a\sigma,A_1\Psi_1,\ldots,A_K\Psi_K), \]
which we write concisely as $g : \Sigma^{1/2} \mapsto A \Sigma^{1/2}$. 
An estimator, $\hat\Sigma^{1/2}= (\hat \sigma , \hat\Psi_1,\ldots, \hat \Psi_K) $, mapping the sample space  $\mathbb{R}^{p_1\times\cdots\times p_K \times n}$ to the parameter space $\mathcal{G}_{\mathbf{p}}^+$ is equivariant if
\[ \hat\Sigma^{1/2} \left(aX \times \{A_1,\ldots,A_K,I_n\}\right) =  (a,A_1,\ldots,A_K)\hat\Sigma^{1/2}\left(X\right). \]
For example, if $\hat{\Psi}_k$ is the estimator of $\Psi_k$ when observing $X$, then $A_k\hat{\Psi}_k$ is the estimator when observing $aX \times \{A_1\ldots,A_K,I_n\}$. 

Unlike the case for the multivariate normal model, 
the class of $\mathcal{G}_{\mathbf{p}}^+$- equivariant estimators 
for the array normal model 
is not easy to  characterize beyond the 
definition given above.  However, in cases like the present one 
where the group space and parameter space are the same, the UMREE 
under an invariant loss 
can be obtained as the 
 generalized Bayes decision rule
under a (generally improper) prior obtained from a right invariant (Haar) measure over the  group 
\citep{zidek1969representation,eaton1989group}. 
The first step towards obtaining the UMREE is then to obtain a 
right invariant measure and corresponding prior. 
To do this, we first need to define an appropriate measure space for the  
elements 
of $\mathcal{G}_{\mathbf{p}}^+$. 
Recall that matrices $A_k$ in $\mathcal G_{p_k}^+$ have determinant 1, 
and so one of the nonzero elements of $A_k$ can be 
expressed as a function of the others. 
For the rest of this section and the next, 
we parameterize $A_k\in\mathcal G_{p_k}^+$
in terms of the elements $\{ A_{k[i,j]} : 2\leq i\leq p_k, 1\leq j \leq i\}$, 
and express the upper-left element $A_{k[1,1]}$ as a function 
of the other diagonal elements, so that $A_{k[1,1]} = \prod_{i = 2}^{p_k}(A_{k[i,i]})^{-1}$. 
The ``free'' elements of $A_k\in \mathcal G_{p_k}^+$ therefore take values in 
the space $\mathcal A_{p_k} = \{  a_{i,i}>0,  
   a_{i,j} \in \mathbb R :  2\leq  i \leq  p_k, 1\leq j < i \}$. 
\begin{theorem}
  \label{theorem:haar.measure}
  A right invariant measure over the group $\mathcal{G}_{\mathbf{p}}^+$ is
  \begin{align*}
    d\nu_r\left(a,A_1,\ldots,A_K \right) = \frac{1}{a} \left ( \prod_{k=1}^{K} \prod_{i=2}^{p_k}A_{k[i,i]}^{i - 2}\right )  \  d\mu\left(a,A_1,\ldots,A_K \right),
  \end{align*}
  where $d\mu$ is Lebesgue measure over 
 $\mathbb R^+ \times \mathcal A_{p_1} \times \cdots \times \mathcal A_{p_K}$.
\end{theorem}
We note that although the density given above is specific to the particular 
parameterization of the $\mathcal G_{p_k}^+$'s, the inference results that 
follow
will hold for any parameterization. 

Let $L: \mathcal{G}_{\mathbf{p}}^+ \times \mathcal{G}_{\mathbf{p}}^+ \rightarrow \mathbb{R}^+$ be an invariant loss function, so that 
$ L(  \Sigma^{1/2}, B ) = 
  L(  A\Sigma^{1/2}, AB) $
for all $A$, $B$ and $\Sigma^{1/2} \in \mathcal G_{\mathbf p_k}^+$. 
Theorem 6.5 of \citet{eaton1989group} implies that 
the value of the  UMREE  when the array $X$ is observed 
is the minimizer in $B=( b,B_1,\ldots, B_K)$  of the integral 
\[
\int_{\mathcal{G}_{\mathbf{p}}^+} L( A \Sigma^{1/2}_0  ,  B )
\times p(X| A \Sigma^{1/2}_0 ) \ d\nu_r (A), 
\]
where 
$p(X| A\Sigma^{1/2}_0)$ is the array normal density at the parameter value 
 $A\Sigma^{1/2}_0$ and 
$\Sigma_0^{1/2}$ is an arbitrary element of $\mathcal G_{\mathbf p}^+$. 
Since the group action is transitive over the parameter space, and since the integral is right invariant, 
$\Sigma_0^{1/2}$ can be chosen 
to be equal to $(1,I_{p_1} ,\ldots, I_{p_K})$. 
Furthermore, since the parameter space and group space are the same, 
replacing $A$ with $\Sigma^{1/2}$ in the above integral indicates that
the UMREE  at $X$ is the minimizer in $B$ of 
\[ 
\int_{\mathcal{G}_{\mathbf{p}}^+} L( \Sigma^{1/2} ,  B )
\times p(X|  \Sigma^{1/2} ) \ d\nu_r (\Sigma^{1/2}), 
\]
that is, the UMREE is the Bayes estimator under the  (improper) prior 
$\nu_r$ for $\Sigma^{1/2}$.  This is summarized in the following corollary:
\begin{corollary}
\label{cor:umree}
For an invariant loss function $L: \mathcal{G}_{\mathbf{p}}^+ \times \mathcal{
G}_{\mathbf{p}}^+ \rightarrow \mathbb{R}^+$ 
the estimator $\hat \Sigma^{1/2}$, defined as 
\begin{equation} 
 \hat \Sigma^{1/2}(X)  = \argmin_{B\in  \mathcal{G}_{\mathbf{p}}^+} 
 \Exp{ L( \Sigma^{1/2} , B ) |X  } , 
\label{equation:posterior.loss}
\end{equation}
uniformly minimizes the risk 
$\Exp{ L(\Sigma^{1/2} , \tilde \Sigma^{1/2}(X))|\Sigma^{1/2}}$ 
among equivariant estimators $\tilde \Sigma^{1/2}$ of $\Sigma^{1/2}$. 
The expectation in (\ref{equation:posterior.loss}) is with respect to the posterior density  
\begin{align}
    \label{equation:equi.posterior}
    p&(\sigma,\Psi_1,\ldots,\Psi_K|X)
    \propto  \\ &  \sigma^{-np} \exp\left\{-\frac{1}{2\sigma^2}||X \times \{ \Psi_1^{-1},\ldots,\Psi_K^{-1},I_n \} ||^2 \right\}
      \frac{1}{\sigma}\prod_{k=1}^K \prod_{i=2}^{p_k}\Psi_{k[i,i]}^{i - 2}, \nonumber 
\end{align}
where $p=\prod_1^K p_k$.
\end{corollary}

In addition to uniformly minimizing the risk, 
the UMREE has two additional features. First, 
since any unique MLE is equivariant \citep[Theorem 3.2]{eaton1989group}, 
the UMREE dominates any unique MLE, presuming the UMREE is not the MLE. 
Second, the UMREE under $\mathcal{G}_{\mathbf{p}}^+$ is minimax. This follows because $\mathcal{G}_{\mathbf{p}}^+$ is a subgroup of $G_{p}^+$, as $a  (A_K \otimes  \cdots \otimes A_1) \in G_{p}^+$ for all $a > 0$ and $A_k \in \mathcal{G}_{p_k}^+$. Since $G_{p}^+$ is a solvable group \citep{james1961estimation}, this necessarily implies that $\mathcal{G}_{\mathbf{p}}^+$ is solvable \citep[Theorem 5.15]{rotman1995introduction}. By the results of \cite{kiefer1957invariance} and \cite{bondar1981amenability}, the equivariant estimator that minimizes (\ref{equation:posterior.loss}) is minimax. 

Note that 
because the prior $\nu_r$ is improper, the posterior (\ref{equation:equi.posterior}) is not guaranteed to be proper. However, we are able to guarantee 
propriety 
if the sample size $n$ is sufficiently large:
\begin{theorem}
\label{theorem:proper.posterior}
Let $n > \prod_{k=1}^K p_k$. For $p(\sigma,\Psi_1,\ldots,\Psi_K|X)$ defined in (\ref{equation:equi.posterior}),
\begin{align*}
\int_{\mathbb{R}^+\times\mathcal{G}_{p_1}^+\times\cdots\times \mathcal{G}_{p_K}^+} p(\sigma,\Psi_1,\ldots,\Psi_K|X)d\sigma d\Psi_1\cdots d\Psi_K < \infty
\end{align*}
\end{theorem}
The sample size in the Theorem is sufficient for propriety, but empirical 
evidence suggests that it is not necessary. For example, results from a simulation study in 
Section 4 suggest that, 
for some dimensions,  
a sample size of $n=1$ is sufficient for 
posterior propriety and existence of an UMREE. 

\section{Posterior approximation}
\label{section:equi.proc}
For the results in Section \ref{section:invariant.measure} to be of use, we must be able to actually 
minimize the posterior risk in Equation \ref{equation:posterior.loss}
under an invariant loss function of interest. 
In the next section, we will show that the posterior risk minimizer under 
a multiway generalization of Stein's loss is given by posterior expectations 
of the form  
$\Exp{ (\sigma^2 \Sigma_k)^{-1}|X}$, where $\Sigma_k = \Psi_k\Psi_k^T$. 
Although these posterior expectations are not generally available 
in analytic form,
they can be approximated using a MCMC algorithm. 
In this section, we show how a relatively simple Gibbs sampler 
can be used to simulate a Markov chain of values of $\Sigma^{1/2} = (\sigma , \Psi_1,\ldots, \Psi_K)$,  having a stationary distribution equal to the desired 
posterior distribution given by Equation \ref{equation:equi.posterior}. 
These simulated values can be used to approximate the 
posterior distribution of $\Sigma^{1/2}$ given $X$, 
as well as any  posterior expectation, 
in particular $\Exp{(\sigma^2\Sigma_k)^{-1}|X}$.

The Gibbs  sampler  proceeds by iteratively simulating
values of $\{ \sigma, \Psi_k\}$ from their full conditional distribution 
given the current values of $\{ \Psi_1,\ldots,\allowbreak \Psi_{k-1} , \Psi_{k+1},\ldots, \Psi_{K}\}$. This is done by simulating $\sigma^2 \Sigma_k$ from 
its full conditional distribution, from which 
$\sigma$ and $\Psi_k$ can be recovered.  One iteration of the Gibbs 
sampler proceeds as follows: 

\smallskip

Iteratively for each $k \in\{ 1,\ldots, K\}$, 
\begin{enumerate}
\item  simulate $(\sigma^2 \Sigma_k)^{-1} \sim \mbox{mirror-Wishart}_{p_k}(np/p_k,(X_{(k)}\Psi_{-k}^{-T}\Psi_{-k}^{-1} X_{(k)}^T)^{-1})$; 
\item set $\Psi_k$ to be the lower triangular Cholesky square root of $\Sigma_k$.
\end{enumerate}
In this algorithm, $X_{(k)} \in \mathbb R^{p_k\times np/p_k }$ is 
the mode-$k$ matricization of $X$ and 
$\Psi_{-k} = \Psi_K \otimes \cdots \otimes \Psi_{k+1} \otimes  \Psi_{k-1} \otimes  \cdots \otimes \Psi_1$. The mirror-Wishart  distribution is a 
 probability distribution on positive definite matrices, 
related to the Wishart  distribution as follows:
\begin{definition} A random $q\times q$ positive definite matrix 
$S$ has a mirror-Wishart distribution with degrees of freedom $\nu > 0$ and scale matrix $\Phi\in \mathcal S_q^+$ if 
\begin{align*}
S \overset{d}{=} U V^TV U^T,
\end{align*}
where $VV^T$ is the lower triangular Cholesky decomposition of a $\mbox{Wishart}_{q}(\nu,I_{q})$-distributed random matrix and 
$UU^T$ is the upper triangular Cholesky decomposition of $\Phi$.   
\end{definition}


Some understanding of the mirror-Wishart distribution can be obtained 
from its expectation: 
\begin{lemma}
\label{lemma:e.sw}
If $S\sim \text{mirror-Wishart}_q(\nu,\Phi)$ then 
\[ \Exp{ S } = \nu  U D U^T  \] 
where $UU^T$ is the upper triangular Cholesky decomposition of $\Phi$ and 
 $D$ is a diagonal matrix with entries 
 $d_j = (\nu+q+1-2j)/\nu, \ j=1,\ldots, q$. 
\end{lemma}
The calculation follows from Bartlett's decomposition, and is in the appendix. 
The implications of 
this for covariance estimation are best understood in the context 
of 
the multivariate normal model $X \sim N_{p\times n}( 0  , I_n\otimes \Sigma)$. 
In this case, for a given prior the Bayes estimator 
under Stein's loss is given by $\Exp{ \Sigma^{-1} | X }^{-1} $
(see, for example \citet{yang_berger_1994}). Under Jeffreys' 
noninformative prior, $\Sigma^{-1} \sim \text{Wishart}_p(n, (XX^T)^{-1})$ 
and so the Bayes estimator is  $XX^T/n$. While unbiased, this estimator 
is generally thought of as not providing appropriate shrinkage of the  sample
eigenvalues. 
Note that
under Jeffreys' prior, 
 \emph{a posteriori}  we have
$\Sigma^{-1} \stackrel{d}{=} U VV^T U^T$, where  
$VV^T\sim \text{Wishart}_p(n,I_p)$ and $UU^T$ is the upper triangular Cholesky decomposition of $(XX^T)^{-1}$.  In contrast, under a right invariant measure as our prior we have 
$\Sigma^{-1} \stackrel{d}{=} U V^T V U^T$. The expectation of  
$VV^T$ is $nI$, whereas the expectation of $V^TV$ is 
  $nD$, which provides a different pattern of shrinkage of the eigenvalues of
$XX^T$.  By Lemma \ref{lemma:e.sw} , the Bayes estimator under a right invariant measure as our prior 
in this case is given by $(nUDU^T)^{-1}  = U^{-T} D^{-1} U^{-1}/n$, 
which is the UMREE 
obtained by \citet{james1961estimation}. 
Thus, the UMREE in the multivariate 
normal model corresponds to a Bayes estimator under a right invariant measure as our prior and  mirror-Wishart posterior distribution.

The Gibbs sampler is based on the full conditional distribution of 
$(\sigma^2 \Sigma_{k})^{-1}$, which we derive from the 
full conditional density of $\{\sigma,\Psi_k\}$:
\begin{align*}
p(\sigma,\Psi_k)   \propto& |\sigma\Psi_k|^{-(np+1)/p_k} \exp \left\{- \tr \left((\sigma^2\Psi_k\Psi_k^T)^{-1}X_{(k)}\Psi_{-k}^{-T}\Psi_{-k}^{-1}X_{(k)}^T\right)/2\right\}\prod_{i=2}^{p_k}\Psi_{k[i,i]}^{i - 2}, 
\end{align*}
where dependence of the density on $\{\Psi_1,\ldots,\Psi_{k-1},\Psi_{k+1},\ldots,\Psi_K,X \}$ has been made implicit. 
Now set $L_k = \sigma \Psi_k$. 
The full conditional density  of $L_k$ can be obtained from that of 
$\{\sigma,\Psi_k\}$ and the Jacobian of the transformation. 
\begin{lemma}
  \label{lemma:det.to.lower}
The Jacobian of the transformation $g(\sigma ,\Psi_k) = \sigma \Psi_k$, 
mapping $\mathbb R^+\times \mathcal G_{p_k}^+$ to $G_{p_k}^+$ is 
  \[
  J(\sigma,\Psi_k) \propto \sigma^{p_k(p_k+1)/2 - 1} \Psi_{k[1,1]}.
  \]
\end{lemma}
Since $L_k = \sigma \Psi_k$, we have $\sigma = |L_k|^{1/p_k}$ and $\Psi_{k[i,i]} = L_{k[i,i]}/\sigma = L_{k[i,i]}/|L_k|^{1/p_k}$. Lemma \ref{lemma:det.to.lower} implies
\begin{align*}
p(L_k)  & \propto |L_k^T|^{-(np+1)/p_k} \exp \left\{-\tr \left((L_kL_k^T)^{-1}X_{(k)}\Psi_{-k}^{-T}\Psi_{-k}^{-1}X_{(k)}^T\right)/2\right\}\\
  &\times \prod_{i=2}^{p_k}\left(L_{k[i,i]}/|L_k|^{1/p_k} \right)^{i - 2} \left(|L_k|^{1/p_k}\right)^{-p_k(p_k + 1)/2 + 1}\left(L_{k[1,1]}/|L_k|^{1/p_k}\right)^{-1},
\end{align*}
which, through straightforward calculations, can be shown to be proportional to
\begin{align*}
  \left(\prod_{i=1}^{p_k}L_{k[i,i]}^{i - np/p_k - p_k - 1}\right)\exp \left\{- \tr \left((L_kL_k^T)^{-1}X_{(k)}\Psi_{-k}^{-T}\Psi_{-k}^{-1}X_{(k)}^T\right)/2\right\}.
\end{align*}

We  now ``absorb'' $X_{(k)}\Psi_{-k}^{-T}\Psi_{-k}^{-1}X_{(k)}^T$ into $L_k$. First, take the lower triangular Cholesky decomposition of $X_{(k)}\Psi_{-k}^{-T}\Psi_{-k}^{-1}X_{(k)}^T = \Phi_k \Phi_k^T$ so that 
\[\left(X_{(k)}\Psi_{-k}^{-T}\Psi_{-k}^{-1}X_{(k)}^T\right)^{-1} = \Phi_k^{-T} \Phi_k^{-1}.\] We have
\begin{align*}
  p(L_k) \propto& \left(\prod_{i=1}^{p_k}L_{k[i,i]}^{i - np/p_k - p_k - 1}\right)\exp \left\{- \tr \left((L_kL_k^T)^{-1}\Phi_k\Phi_k^T\right)/2\right\} \\
  \propto& \left(\prod_{i=1}^{p_k}L_{k[i,i]}^{i - np/p_k - p_k - 1}\right)\exp \left\{- \tr \left((\Phi_k^{-1}L_k(\Phi_k^{-1}L_k)^T)^{-1}\right)/2\right\}.
\end{align*}
Now let $W_k = \Phi_k^{-1}L_k$, so that $L_k = \Phi_kW_k$. This change of variables has Jacobian $J(W_k) = \prod_{i = 1}^{p_k} \Phi_{k[i,i]}^{i}$ \citep[Proposition 5.13]{eaton1983multivariate}, so that
\begin{align}
  \label{equation:w.marginal}
  p(W_k) \propto \left(\prod_{i=1}^{p_k}W_{k[i,i]}^{i - np/p_k - p_k - 1}\right)\exp \left\{- \tr \left((W_kW_k^T)^{-1}\right)/2\right\}.
\end{align}
Note that the distribution of $W_k$ does not depend on $\Psi_{-k}$. 
Now 
compare equation (\ref{equation:w.marginal}) to the density of the lower triangular Cholesky square root $W$ of an inverse-Wishart distributed random matrix \[WW^T \sim \mbox{inverse-Wishart}_{p_k}\left(np/p_k,I_{p_k}\right),\]  given by
\begin{align}\
\label{equation:chol.inv.wish}
  p(W) \propto \left(\prod_{i = 1}^{p_k} W_{[i,i]}^{-np/p_k - i}\right) \exp\left\{ - \tr((WW^T)^{-1}) /2\right\}.
\end{align}
The conditional densities of the off-diagonal elements of $W_k$ and $W$ given the diagonal elements 
 clearly have the same form. The diagonal elements of $W_k$ and $W$ in (\ref{equation:w.marginal}) and (\ref{equation:chol.inv.wish}) turn out to be square roots of inverse-gamma distributed random variables, but with different shape parameters. To show this, we first derive the conditional densities of the
off-diagonal elements of $W$:
\begin{lemma} (Bartlett's decomposition for the inverse-Wishart)
  \label{lemma:bartlett.inv.wishart} 
Let $W$ be the lower triangular Cholesky square root of an inverse-Wishart distributed matrix, so 
$WW^T\sim \mbox{inverse-Wishart}_{p_k}(\nu,I_{p_k})$.
Then for each $i=1,\ldots,p_k$, 
  \begin{align*} 
&     W_{[i,i]}^2  \sim \mbox{ inverse-gamma} ( [\nu - p_k + i]/2, 1/2 ), \mbox{ and}\\
&     W_{[i,1:(i-1)]}|W_{[i,i]}, W_{[1:(i-1),1:(i-1)]}\\
&     \sim N_{i-1}\left( 0, W_{[i,i]}^2W_{[1:(i-1),1:(i-1)]}^TW_{[1:(i-1),1:(i-1)]} \right).
  \end{align*}
\end{lemma}
Here, $W_{[ 1:(i-1),1:(i-1)]}$ denotes the submatrix of $W$ made up of the 
first $(i-1)$ rows and columns, and $W_{[i,1:(i-1)]}$ is the vector made up 
of the first $(i-1)$ elements of the $i$th row.

By Lemma \ref{lemma:bartlett.inv.wishart}, if 
$WW^T\sim \text{inverse-Wishart}(np/p_k,I_{p_k}  )$ then 
the squared diagonal elements of $W$  are independent $\mbox{inverse-gamma} ( (n p/p_k - p_k + i)/2, 1/2 )$  random variables. 
This tells us that 
\begin{align*}
\int  \exp\left\{ - \tr((WW^T)^{-1}) /2 \right\} 
\prod_{i>j} dW_{[i,j]} 
 \ \propto \  \prod_{i = 1}^{p_k} W_{[i,i]}^{p_k - 1} \exp\left\{-1/(2W_{[i,i]}^2)\right\}.
\end{align*}
This result allows us to integrate
(\ref{equation:w.marginal}) with respect to the off-diagonal elements of $W_k$, giving
\begin{align*}
&\int\left(\prod_{i=1}^{p_k}W_{k[i,i]}^{i - np/p_k - p_k - 1}\right)\exp \left\{-\tr \left((W_kW_k^T)^{-1}\right)/2\right\}
\prod_{i>j} dW_{[i,j]}  \\
& \propto W_{k[i,i]}^{i - np/p_k  - 2} \exp\left\{-1/(2W_{k[i,i]}^2)\right\}. 
\end{align*}
A  change of variables implies that the $W_{k[i,i]}^2$'s are independent, and
\begin{align}
  W_{k[i,i]}^2 \sim \mbox{ inverse-gamma}( [ np/p_k - i + 1]/2, 1/2 ). 
\label{eq:wkdiag}
\end{align}
This completes the characterization of the distribution of $W_k$: 
The distribution of the diagonal elements is given by 
(\ref{eq:wkdiag}) and the conditional distribution of the 
off-diagonal elements given the diagonal can be obtained from 
Lemma \ref{lemma:bartlett.inv.wishart}. 
Finally, this distribution can be related to a Wishart distribution via 
the following lemma:
\begin{lemma}
  \label{lemma:lt.wishart.inverse}
  Let $W_k$ be a random $p_k\times p_k$ lower triangular matrix such that
  \begin{align*}
    &W_{k[i,i]}^2 \sim \mbox{ inverse-gamma} \left( [\nu - i + 1]/2, 1/2 \right), \mbox{ and}\\
     &W_{k[i,1:(i-1)]}|W_{k[1:(i-1),1:(i-1)]},W_{k[i,i]} \\
& \sim N_{i-1}\left( 0, W_{k[i,i]}^2W_{k[1:(i-1),1:(i-1)]}^TW_{k[1:(i-1),1:(i-1)]} \right).
  \end{align*}
  Then the elements of $V_k = W_k^{-1}$ are distributed independently as 
  \begin{align*} 
     & V_{k[i,i]}^2 \sim \text{gamma} ([\nu - i + 1]/2, 1/2 ) , \ i=1,\ldots, q  \\ 
   &  V_{k[i,j]} \sim N(0,1),  \ i\neq j. 
  \end{align*}
\end{lemma}
Note that the matrix $V_k$ is distributed as the lower triangular Cholesky square root of a Wishart distributed random matrix. 
Applying the lemma to $W_k$, for which $\nu = np/p_k$, we have 
that 
$V_k = W_k^{-1} = (\Phi_k^{-1}L_k )^{-1} = L_k^{-1}\Phi_k = \frac{1}{\sigma}\Psi_k^{-1}\Phi_k$ is equal in distribution to the lower triangular Cholesky square root of a $\mbox{Wishart}_{p_k}(np/p_k,I_{p_k})$ random matrix. That is, the precision matrix $(\sigma^2 \Psi_k \Psi_k^T )^{-1} = \Psi_k^{-T}\Psi_k^{-1}/\sigma^2$ is conditionally distributed as
\begin{align*}
  &\frac{1}{\sigma^2}\Psi_k^{-T}\Psi_k^{-1}|\Psi_{-k} \overset{d}{=} \Phi_k^{-T} V^TV \Phi_k^{-1}, \mbox{ where}\\
  & VV^T \sim \mbox{Wishart}_{p_k}\left(np/p_k,I_{p_k}\right) \mbox{ and } \Phi_k\Phi_k^T = X_{(k)}\Psi_{-k}^{-T}\Psi_{-k}^{-1}X_{(k)}^T.
\end{align*}
We say the matrix, $\Phi_k^{-T} V^TV \Phi_k^{-1}$ has a \emph{mirror-Wishart} distribution because $\Phi_k^{-T} VV^T \Phi_k^{-1}$ would have a Wishart distribution. 
This completes the derivation of the full conditional distribution of 
$\sigma^2 \Sigma_k = \sigma^2 \Psi_k\Psi_k^T$.

Although not necessary for posterior approximation, the full conditional distribution of $\sigma$ given $\Psi_1,\ldots,\Psi_K$ and $X$ is easy to derive. The posterior density is
\[
p(\sigma) \propto \sigma^{-(np+1)}\exp\left\{-\left|\left|X \times \{ \Psi_1^{-1},\ldots,\Psi_K^{-1},I_n \} \right|\right|^2/(2\sigma^2) \right\}.
\]
Letting $\gamma = 1/\sigma^2$, we have
\[
p(\gamma) \propto \gamma^{np/2 - 1}\exp\left\{-\gamma\left|\left|X \times \{ \Psi_1^{-1},\ldots,\Psi_K^{-1},I_n \} \right|\right|^2/2 \right\},
\]
and so the full conditional distribution of $1/\sigma^2$ is \[\mbox{gamma}(np/2,||X \times \{ \Psi_1^{-1},\ldots,\Psi_K^{-1},I_n \} ||^2/2).\]

\section{Estimation under multiway Stein's loss}
\subsection{The UMREE for multiway Stein's loss}
\label{section:mree}

A commonly used loss function for estimation  of a covariance matrix $\Sigma$ is Stein's loss, 
 \begin{align*}
L_S\left(S,\Sigma\right) = \tr\left(S\Sigma^{-1}\right) - \log\left|S\Sigma^{-1}\right| - p   ,  \  \ \Sigma, S \in \mathcal S_{p}^+.
\end{align*}
First introduced by \cite{james1961estimation}, Stein's loss has been proposed as a reasonable and perhaps better alternative to quadratic loss for 
evaluating performance of covariance estimators. 
For example, Stein's loss, unlike  quadratic loss, does not  penalize overestimation of the variances  more severely than underestimation. 

Recall from Section \ref{section:invariant.measure} that the array normal model can be parameterized in terms  of $\Sigma = (\sigma^2,\Sigma_1,\ldots, \Sigma_K) \in \mathcal S_{\bf p}^+$, 
where $|\Sigma_k|=1$ for each $k=1,\ldots, K$. 
For estimation 
of the covariance parameters  $\Sigma\in \mathcal S^{+}_{\bf p}$, 
we consider the following generalization of Stein's loss, which we call 
``multiway Stein's loss'':
\begin{align}
\begin{split}
  \label{equation:multi.stein.loss}
  L_M&\left(\Sigma,S\right) = \frac{s^2}{\sigma^2} \sum_{k=1}^{K} \frac{p}{p_k} \tr\left[ S_k\Sigma_k^{-1}\right] - Kp\log\left( \frac{s^2}{\sigma^2} \right) - Kp    ,  \  \ \Sigma, S \in \mathcal S_{\bf p}^+. 
\end{split}
\end{align}
It is easy 
to see that for $K=1$, multiway Stein's loss reduces to Stein's loss. 
Multiway Stein's loss also has the attractive property of being invariant under multilinear transformations. To see this, define $SL_{\bl p }$  
to be the set of lists of the form $A= (a,A_1,\ldots, A_K)$ for which $a>0$  and $A_k \in SL_{p_k}$ for each $k$, with
$SL_{p_k}$ being the special linear group of 
$p_k\times p_k$ matrices with unit determinant. 
For two elements $A$ and $B$ of $SL_{\bl p}$, define 
$AB = (ab, A_1 B_1,\ldots, A_KB_K)$ and 
$A^T = (a, A_1^T,\ldots, A_K^T)$. 
Multiway Stein's loss is invariant under transformations
of the form $\Sigma \rightarrow  A \Sigma A^T$, as 
\begin{align*}
\begin{split}
  &L_M\left(A\Sigma A^T,A S A^T\right)\\
  &= \frac{a^2 s^2}{a^2\sigma^2} \sum_{k=1}^{K} \frac{p}{p_k} \tr\left[ A_kS_k  A_k^T \left(A_k\Sigma_k A_k^T\right)^{-1}\right] - Kp\log\left( \frac{a^2 s^2}{a^2 \sigma^2} \right) - Kp\\
  &= \frac{s^2}{\sigma^2} \sum_{k=1}^{K} \frac{p}{p_k} \tr\left[ S_k\Sigma_k^{-1}\right] - Kp\log\left( \frac{s^2}{\sigma^2} \right) - Kp\\
  &= L_M\left(\Sigma,S\right).
\end{split}
\end{align*}
In particular, 
(\ref{equation:multi.stein.loss}) is invariant under $\mathcal{G}_{\bf p}^+$, 
as $\mathcal G_{\bl p}^+\subset SL_{\bl p}$. 
Therefore, the best $\mathcal G_{\bf p}^+$-equivariant 
estimator under multiway Stein's loss can be obtained using 
Corollary \ref{cor:umree}.

\begin{proposition}{(UMREE under multiway Stein's loss)}
  \label{prop:mree}
  Let \[\mathcal{E}_k = \left( E\left[\left. \left(\sigma^2\Sigma_k\right)^{-1}\right|X\right]\right)^{-1},\] where the expectation is with respect to the posterior distribution given by  Equation \ref{equation:equi.posterior}. The minimizer of the posterior expectation
  \begin{align*}
    E\left[\left.\frac{s^2}{\sigma^2} \sum_{k=1}^{K} \frac{p}{p_k} \tr\left[ S_k^T \Sigma_k^{-1}\right] - Kp\log\left( \frac{s^2}{\sigma^2} \right) - Kp \right|X\right]
  \end{align*}
  with respect to $s$ and the $S_k$'s is
  \begin{align*}
    \hat{\Sigma}_k &= \mathcal{E}_k / |\mathcal{E}_k|^{1/(p_k)}\\
    \hat{\sigma}^2 &= \left(\sum_{k=1}^{K}\frac{1}{K} |\mathcal{E}_k|^{-1/p_k}\right)^{-1}.
  \end{align*}
\end{proposition}
The posterior expectation $E[(\sigma^2\Sigma_k)^{-1}|X]$ may be approximated by the Gibbs sampler of Section \ref{section:equi.proc}. That is, if  
$(\sigma^2 \Sigma_k)^{(1)},\ldots, \allowbreak 
 (\sigma^2 \Sigma_k)^{(T)}$ is a long sequence of values of 
$(\sigma^2 \Sigma_k)$ simulated  from the Gibbs sampler, then 
\[
{\rm E}[(\sigma^2\Sigma_k)^{-1}|X] \approx \sum_{t=1}^T [(\sigma^2\Sigma_k)^{(t)}]^{-1}/T.
\]


The form of multiway Stein's loss (\ref{equation:multi.stein.loss}) 
includes a weighted sum of ${\rm tr}(S_k \Sigma_k^{-1})$, $k=1,\ldots, K$. 
We note that equivariant estimation of $\Sigma$
is largely unaffected by changes to the weights in this sum:
\begin{proposition}
Define weighted multiway Stein's loss as
\begin{align*}
\begin{split}
  &L_W\left(\Sigma,S\right) = \frac{s^2}{\sigma^2} \sum_{k=1}^{K} \frac{w_k}{p_k} \tr\left[ S_k\Sigma_k^{-1}\right] - \left(\sum_{k = 1}^{K}w_k\right)\log\left( \frac{s^2}{\sigma^2} \right) - \sum_{k=1}^{K}w_k,
\end{split}
\end{align*}
for known $w_k > 0$, $k = 1,\ldots,K$. Then the UMREE under $L_W$
is given by
\begin{align*}
    \hat{\Sigma}_k &= \mathcal{E}_k / |\mathcal{E}_k|^{1/(p_k)}\\
    \hat{\sigma}^2 &= \left(\sum_{k=1}^{K} \frac{w_k}{\sum_{i = 1}^{K}w_i}|\mathcal{E}_k|^{-1/p_k} \right)^{-1}. 
\end{align*}
\end{proposition}
The proof is very similar to that of Proposition \ref{prop:mree} and is omitted.
This proposition states that only estimation of the scale is affected when we ``weight'' the loss more heavily for some
components of $\Sigma$ than others.

The posterior distribution may also be used to obtain the UMREE under Stein's original loss $L_S$, as it too is invariant under transformations of the lower triangular product group. However, risk minimization with respect to $L_S$ requires additional numerical approximations:
Let $\mathcal{K}$ be the unique symmetric square root of $E[(\Sigma_K^{-1} \otimes \cdots \otimes \Sigma_1^{-1})/\sigma^2|X]$. This $\mathcal{K}$ may be approximated by the Gibbs sampler described in Section \ref{section:equi.proc}. Minimization of the risk with respect to $L_S$ is equivalent to the minimization in $(s^2,S_1,\ldots,S_K)$ of
\begin{align*}
E[L_S(S,\Sigma)|X] & = s^2\tr\left(\mathcal{K}\left(S_K\otimes\cdots\otimes S_1\right)\mathcal{K}\right) - p\log\left(s^2\right) + c(\Sigma)\\
&=s^2\left|\left|\tilde{\mathcal{K}} \times \left\{S_1^{1/2},\ldots,S_K^{1/2},I_p\right\}\right|\right|^2 - p\log\left(s^2\right) + c(\Sigma)\\
&=\tr\left(s^2S_k\tilde{\mathcal{K}}_{(k)}S_{-k}\tilde{\mathcal{K}}_{(k)}^T\right) - p\log\left(|s^2S_k|\right)/p_k + c(\Sigma),
\end{align*}
where $\tilde{\mathcal{K}} \in \mathbb{R}^{p_1\times\cdots\times p_K \times p}$ is the array such that $\tilde{\mathcal{K}}_{(K+1)} = \mathcal{K}$, and $S_k^{1/2}$ is any square root matrix of $S_k$. Iteratively setting $s^2S_k = (\tilde{\mathcal{K}}_{(k)}S_{-k}\tilde{\mathcal{K}}_{(k)}^T)^{-1}p/p_k$ will decrease the posterior expected loss at each step. 
This procedure is analogous to using the iterative flip-flop algorithm to 
find the MLE based on a sample covariance matrix of $E[(\Sigma_K^{-1} \otimes \cdots \otimes \Sigma_1^{-1})/\sigma^2|X]$. Application of the results from \citep{wiesel2012geodesic}
show that the posterior risk has a property known as geodesic convexity, implying that 
any local
 minimizer obtained from this algorithm will also be a global minimizer. 

\subsection{An orthogonally equivariant estimator}

\label{section:orthogonal}

The estimator in Proposition \ref{prop:mree} depends on the ordering of the indices, and so it is not permutation equivariant. Mirroring the ideas studied in \cite{takemura1983orthogonally}, in this section we derive a minimax orthogonally equivariant estimator (which is necessarily permutation equivariant) that dominates the UMREE of Proposition \ref{prop:mree}. 
First, notice that 
by transforming the data and then back-transforming the estimator, 
we can obtain an estimator whose risk is equal to that of the UMREE:
For $\Gamma = (1,\Gamma_1,\ldots, \Gamma_K) \in \{1\}\times \mathcal{O}_{p_1}\times\cdots\times \mathcal{O}_{p_K}$, where $\mathcal{O}_{p_k}$ is the group of $p_k$ by $p_k$ orthogonal matrices, let 
$\tilde X =  X \times \{ \Gamma_1,\ldots, \Gamma_K\}$. Then 
$\hat \Sigma (\tilde X)$ is an estimator of $\Gamma\Sigma \Gamma^T$ and 
 $\tilde \Sigma(X) =  \Gamma^{T} \hat \Sigma(\tilde X) \Gamma$ is an estimator of 
$\Sigma$. The risk of this estimator is the same as that of the UMREE
$\hat \Sigma(X)$:
\begin{align*}
  R\left(\Sigma,\tilde \Sigma(X) \right) &= E\left[\left. L_M\left(\Sigma,\Gamma^{T} \hat \Sigma(\tilde{X})\Gamma \right)\right| \Sigma \right]\\
  &= E\left[\left. L_M\left(\Gamma\Sigma \Gamma^T, \hat\Sigma(\tilde{X})  \right)\right|\Sigma\right]\\
  &= E\left[\left. L_M\left(\Gamma \Sigma \Gamma^T, \hat \Sigma(X)  \right)\right| \Gamma \Sigma \Gamma^T\right]\\
  &= R\left(\Gamma\Sigma \Gamma^T,\hat \Sigma (X)\right)\\
  &= R\left(\Sigma,\hat \Sigma(X)\right)
\end{align*}
where the second equality follows from  the invariance of the loss, the third equality follows from a change of variables, and the last equality follows because the risk of $\hat \Sigma$ is constant over the parameter space. 
The UMREE $\hat \Sigma$ and the estimator $\tilde \Sigma$ have the same 
risks but are different.
Since multiway Stein's loss is convex in each argument, averaging these 
estimators somehow should produce a new estimator that dominates them both. 


In the  multivariate normal case in which $K=1$, 
averaging the value of 
$\Gamma^T\hat \Sigma (\Gamma X )\Gamma$
with respect to the uniform (invariant) measure for $\Gamma$ over the 
orthogonal group results in the estimator of \cite{takemura1983orthogonally}. This estimator is orthogonally equivariant, dominates the UMREE and is therefore also minimax.  
Constructing an analogous estimator in the multiway case  is more complicated, 
as it is not immediately clear how the back-transformed 
estimators should be averaged. Direct numerical averaging  
of estimates of $\sigma^2 ( \Sigma_1\otimes \cdots \otimes \Sigma_K)$ 
will generally produce an estimate that is  not separable and therefore outside of the parameter space. Similarly, averaging estimates of each 
$\Sigma_k$ separately will not work, as the space of covariance matrices 
with determinant one is not convex.

Our solution 
to this problem is to average a transformed version of  
$\Sigma= (\sigma^2,\Sigma_1,\ldots, \Sigma_K)$  for which
each $\Sigma_k$ lies in the convex set of trace-1 covariance matrices, then 
transform back to our original parameter space. 
The resulting estimator, which we call the 
multiway Takemura estimator (MWTE), is 
orthogonally equivariant and uniformly dominates the UMREE.
\begin{proposition}
\label{prop:takemura}
Let  $\hat \sigma^2 (\Gamma, X)$ and $\hat \Sigma_k(\Gamma, X)$ be the UMREEs of $\sigma^2$ and  $\Gamma_k\Sigma_k \Gamma_k^T$ 
based on 
 data $X \times \{\Gamma_1,\ldots,\Gamma_K,I_n\}$.
Let
\begin{align*}
  S_k(X) = \int_{\mathcal{O}_{p_K}} \cdots \int_{\mathcal{O}_{p_1}} \frac{\Gamma_k^T \hat \Sigma_k(\Gamma,X)\Gamma_k}{\tr\left(\hat \Sigma_k\left(\Gamma, X\right)\right)} \ d\Gamma_1 \cdots d\Gamma_K
\end{align*}
and
\begin{align*}
  \tilde \sigma^2 (X) = \int_{\mathcal{O}_{p_K}}\cdots \int_{\mathcal{O}_{p_1}} \hat \sigma^2 \left(\Gamma, X\right) \  d\Gamma_1 \cdots d\Gamma_K.
\end{align*}
Let  $\tilde \Sigma_k(X) =  S_k(X)/|S_k(X)|^{1/p_k}$  for $k = 1,\ldots,K$. Then $(\tilde \sigma^2(X),\tilde \Sigma_1(X),\ldots,\allowbreak \tilde \Sigma_K(X))$ is orthogonally equivariant and uniformly dominates the UMREE of Proposition \ref{prop:mree}.
\end{proposition}
Note that ``averaging'' over any subset of $\mathcal{O}_{p_1} \times \cdots \times \mathcal{O}_{p_K}$ in the manner of Proposition \ref{prop:takemura} will uniformly decrease the risk. By averaging with respect to the uniform measure over the orthogonal group, we obtain an estimator that has the attractive property of being orthogonally equivariant.

In practice it is computationally infeasible to integrate over the space of orthogonal matrices. However, we may obtain a stochastic approximation 
to the MWTE as follows: Independently
for each $t=1,\ldots,T$ and $k=1,\ldots, K$, 
simulate  $\Gamma_k^{(t)}$ from the  uniform distribution on  $\mathcal{O}_{p_k}$. Let 
\begin{align*}
  S_k(X) = \frac{1}{T}\sum_{t = 1}^T \frac{{\Gamma_k^{(t)}}^T\hat \Sigma_k\left(\Gamma^{(t)}, X\right)\Gamma_k^{(t)}}{\tr\left(\hat \Sigma_k\left(\Gamma^{(t)}, X\right)\right)} \ , \ 
  \tilde{\sigma}_T(X) = \frac{1}{T}\sum_{t = 1}^T \hat \sigma(\Gamma^{(t)}, X).
\end{align*}
Set  $\tilde\Sigma_{k,T}(X) = S_k(X)/|S_k(X)|^{1/p_k}$ for $k = 1,\ldots,K$. Then an approximation to the MWTE is 
\begin{align}
\label{equation:TMWTE}
\tilde \Sigma_T = \left(\tilde{\sigma}_T^2(X),\tilde{\Sigma}_{1,T}(X),\ldots,\tilde{\Sigma}_{K,T}(X)\right).
\end{align}
This is a randomized estimator which is orthogonally invariant in the sense of Definition 6.3 of \cite{eaton1989group}.

\subsection{Simulation results}

\label{section:sims}
We numerically compared the risks of the MLE, UMREE, and the 
MWTE under several three-way array normal distributions, using a variety of values of $(p_1,p_2,p_3)$ and with $n = 1$. For each $(p_1,p_2,p_3)$ under consideration,  we simulated 100 data arrays from the array normal model. 
As the risk of both the MLE and the UMREE are constant over the parameter 
space, it is sufficient to compare their risks at a single point in the parameter space, which we took to be $\Sigma=(1,I_{p_1},I_{p_2},I_{p_3})$. 
Risks were approximated by averaging the losses of each estimator across 
the 100 simulated data arrays.  For each data array, 
the MLE was obtained from the iterative coordinate descent algorithm outlined in \citep{hoff2011separable}. 
Each UMREE
was approximated based on 1250 iterations of the Gibbs sampler described in Section \ref{section:equi.proc}, from which the first 
250 iterations were discarded to allow for 
convergence to the stationary distribution (convergence appeared to be essentially immediate). 

\begin{figure}
\begin{center}
\includegraphics{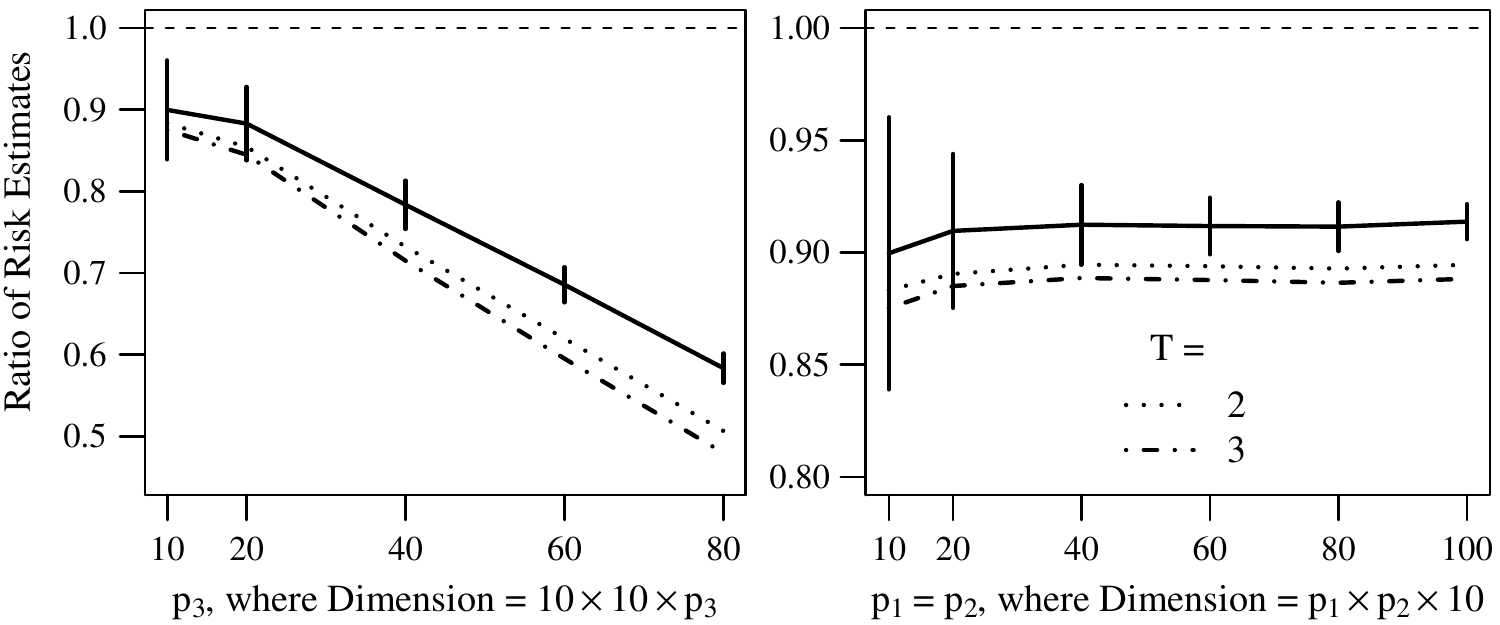}
\caption{Risk comparisons for the MLE, UMREE and MWTE. Both panels 
plot Monte Carlo estimates of the risk ratios of the UMREE to the MLE in solid lines, 
and the approximate MWTE to the MLE in dashed lines. The width of the vertical bars is one standard deviation of the ratio of the UMREE loss to the MLE loss, across 
the 100 data sets. }
\label{fig:risk.ratio}
\end{center}
\end{figure}

The ratio of risk estimates 
across several values of $(p_1,p_2,p_3)$  
are  are plotted in solid lines 
in Figure \ref{fig:risk.ratio}. 
We considered array dimensions in which the first two dimensions 
were identical. This scenario could correspond to, for example, 
data arrays representing 
 longitudinal relational or network  measurements 
between $p_1=p_2$ nodes at $p_3$ time points.
The first panel of the figure considers the relative performance of the 
estimators as the ``number of time points'' ($p_3$) increases. 
The results indicate that 
the UMREE provides substantial and increasing risk improvements compared to  the MLE
as $p_3$ increases. 
However,  the right panel indicates that the gains are not as dramatic 
and not increasing
when the ``number of nodes'' ($p_1=p_2$) increases while $p_3$ remains fixed. 
Even so, the variability in the ratio of losses (shown with vertical bars) decreases as the number of nodes increases, indicating an increasing probability that that the UMREE will beat the MLE in terms of loss. 

We also compared these risks to  the risk of the approximate MWTE given in (\ref{equation:TMWTE}), with $T\in \{2,3\}$. The risks for the approximate 
MWTE relative to those of the MLE are shown 
in dashed lines in the two panels of the Figure, and indicate non-trivial 
improvements in risk as compared to the UMREE. We examined values of $T$
greater than $3$ but found no appreciable further reduction in the risk.
 Note, however, that the MWTE does not have constant risk over the parameter space (though MWTE will have constant risk over the orbits of the orthogonal product group).

\section{Discussion}
\label{section:discussion}
This article has extended the results of \cite{james1961estimation} and \cite{takemura1983orthogonally} by developing equivariant and minimax estimators of the covariance parameters in the array normal model. 
Considering the class of estimators equivariant with respect to a 
special lower triangular group, we showed that the uniform minimum risk equivariant estimator (UMREE) can be viewed as a generalized Bayes estimator 
that can be  obtained from a simple Gibbs sampler.
We obtained an orthogonally equivariant estimator based on this UMREE by 
combining values of the UMREE under orthogonal transformations of the data. 
Both the UMREE and the orthogonally equivariant estimator are minimax, and both 
dominate any unique MLE in terms of risk.

Empirical results in Section 4 indicate that the risk improvements of the UMREE 
over the MLE can be substantial, 
while the improvements of the orthogonally equivariant estimator over the UMREE 
are more modest. However, the risk improvements depend on 
the array dimensions 
in a way that is not currently understood. 
Furthermore, we 
do not yet know the minimal conditions necessary for the 
propriety of the posterior or the existence of the 
UMREE. 
Empirical results from the simulations in Section 4 suggest that 
the UMREE exists for sample sizes as low as $n=1$, at least for the array 
dimensions in the study. 
This is similar to the current state of knowledge for the existence of the 
MLE:
The array normal likelihood is trivially bounded for $n\geq p$  
(as it is bounded by the maximized likelihood under the unconstrained
$p$-variate normal model), and 
some sufficient conditions for uniqueness of the MLE 
are given in  \cite{ohlson2013multilinear}. However, 
empirical results (not shown) suggest that a unique  MLE may exist 
for $n=1$ for some array dimensions (although not for others). 
Obtaining necessary and sufficient conditions for the existence of the UMREE and the MLE 
is an ongoing area of research of the authors.


\appendix

\section{Proofs}

\subsection{Proof of Theorem \ref{theorem:haar.measure}}
\begin{proof}
  Let $a > 0$ , $A_k \in \mathcal{G}_{p_k}^+$ for all $k = 1,\ldots,K$. Let $t$ be a fixed element in $\mathbb{R}^+$ and $T_k$ be fixed elements in $\mathcal{G}_{p_k}^+$ for $k = 1,\ldots,K$. In the terminology of Definition 1.7 of \cite{eaton1989group}, the integral with respect to Lebesgue measure is relatively right invariant with multiplier
  \begin{align*}
    \chi \left(t,T_1,\ldots,T_K\right) = t\prod_{k = 1}^{K}\prod_{i=2}^{p_k}T_{k[i,i]}^{2-i}
  \end{align*}
if the following holds:
  \begin{align}
    \begin{split}
      \label{equation:relatively.invariant}
    &\int_{\mathcal{G}_{\mathbf{p}}^+} f(a/t,A_1T_1^{-1},\ldots,A_KT_K^{-1})d\mu\left(a,A_1,\ldots,A_K\right)\\
    &=\left(t\prod_{k = 1}^{K}\prod_{i=2}^{p_k}T_{k[i,i]}^{2-i}\right)\int_{\mathcal{G}_{\mathbf{p}}^+} f(a,A_1,\ldots,A_K)d\mu\left(a,A_1,\ldots,A_K\right),
  \end{split}
  \end{align}
for arbitrary $f()$. If (\ref{equation:relatively.invariant}) holds, then by Theorem 1.6 of \cite{eaton1989group}, a right invariant measure over the group $\mathcal{G}_{\mathbf{p}}^+$ is
  \begin{align*}
    \chi \left(a,A_1,\ldots,A_K\right)^{-1} = \frac{1}{a}\prod_{k=1}^{K} \prod_{i=2}^{p_k}A_{k[i,i]}^{i - 2} d\mu.
  \end{align*}

  It remains to make a change of variables to show that (\ref{equation:relatively.invariant}) holds. For $E_k, T_k \in \mathcal{G}_{p_k}^+$ with $T_k$ fixed for $k = 1,\ldots,K$, let $g_k(E_k) =  E_k T_k$ for $k = 1,\ldots, K$. For $e,t > 0$ with $t$ fixed let $g(e) = et$. The Jacobian for transforming the scale, $g(e) = et$, is $t$. The Jacobian for the transformation $g_k(E_k) = E_k T_k$ is
  \begin{align}
    \label{equation:gpk.jacobian}
    J(E_k) = \prod_{i=2}^{p_k}T_{k[i,i]}^{2-i}.
  \end{align}
To see this, note that this transformation is equivalent to $p_k(p_k+1)/2 - 1$ linear transformations of the form:
  \begin{align*}
  g_{i,j}: E_{k[i,j]} \mapsto \sum_{j \leq m \leq i} E_{k[i,m]}T_{k[m,j]} \mbox{ for all } 1 \leq j \leq i \leq p_k \mbox{ s.t. } (i,j) \neq (1,1).
  \end{align*}
 Stack the elements of $E_k$ into the following vector:
  \begin{align*} 
    s =& (E_{k[p_k,p_k]}, E_{k[p_k,p_k-1]}, E_{k[p_k-1,p_k-1]}, E_{k[p_k,p_k-2]},\\
& E_{k[p_k-1,p_k-2]}, E_{k[p_k-2,p_k-2]}, E_{k[p_k,p_k-3]}, \ldots, E_{k[2,1]} ),
  \end{align*}
  and notice that the matrix of the linear transformation is lower triangular:
\begin{align*}
&u = \\
  &\left(
    \begin{array}{c@{\hspace{2pt}}c@{\hspace{2pt}}c@{\hspace{2pt}}c@{\hspace{2pt}}c@{\hspace{2pt}}c@{\hspace{2pt}}c}
      T_{k[p_k,p_k]} & 0 & 0 & 0 & \cdots & & \\
      T_{k[p_k,p_k -1]} &  T_{k[p_k - 1,p_k - 1]} & 0 & 0 & & & \\
      0 & 0 & T_{k[p_k - 1,p_k - 1]} & 0 & & &\\
      T_{k[p_k,p_k - 2]} & T_{k[p_k-1,p_k-2]} & 0 & T_{k[p_k - 2,p_k - 2]} & & &\\
      \vdots & & & & \ddots &  & \\
      \vdots & & & & &  \ddots  & \\
      0 & \cdots & \cdots & \cdots & T_{k[2,1]} & \cdots & T_{k[1,1]}
    \end{array}
  \right)
\end{align*}\normalsize
where in the diagonal, each $T_{k[i,i]}$ is repeated $p_k - i + 1$ times for $i = 2,3,\ldots,p_k$, and $T_{k[1,1]}$ is repeated $p_k - 1$ times.
That is, the linear transformation can be written as:
\[
g_k(s) = us.
\]
Hence the determinant of the Jacobian is
\begin{align*}
  |u| = T_{k[1,1]}^{p_k-1} \prod_{i=2}^{p_k}T_{k[i,i]}^{p_k-i+1} = \prod_{i=2}^{p_k}T_{k[i,i]}^{2 - i},
\end{align*}
where the second equality results from our parameterization of $\mathcal{G}_{p_k}^+$, \[\prod_{i=2}^{p_k}T_{k[i,i]}^{-1} = T_{k[1,1]}.\]
\end{proof}

\subsection{Proof of Theorem \ref{theorem:proper.posterior}}
Consider the reformulation of the problem to a parameterization of $\Sigma = \sigma^2(\Psi_K\Psi_K^T\otimes\cdots\otimes\Psi_1\Psi_1^T)$ where $\Psi_{k[1,1]} = 1$ for $k = 1,\ldots,K$. That is, we now work with the group $\mathcal{G}_{\mathbf{p}}^{1} = \{(a,A_1,\ldots,A_K)|a > 0, A_k \in \mathcal{G}_{p_k}^1 \text{ for } k =1,\ldots,K\}$ where $\mathcal{G}_{p_k}^1$ is the group of $p_k$ by $p_k$ lower triangular matrices with positive diagonal elements and $1$ in the $(1,1)$ position. The group operation in $\mathcal{G}_{p_k}^1$ is matrix multiplication, and that of $\mathcal{G}_{\mathbf{p}}^1$ is component-wise multiplication. The left and right Haar measures over $\mathcal{G}_{p_k}^1$ are easy to derive:
\begin{lemma}
  \label{lemma:jacob.left.right}
  For $E_k,T_k \in \mathcal{G}_{p_k}^1$ with $T_k$ fixed, the Jacobian for the transformation $g(E_k) = E_kT_k$ is
  \begin{align*}
    J(E_k) = \prod_{i=2}^{p_k}T_{k[i,i]}^{p_k - i + 1}
  \end{align*}
  the Jacobian for the transformation $g(E_k) = T_kE_k$ is
  \begin{align*}
    J(E_k) = \prod_{i=2}^{p_k}T_{k[i,i]}^{i}
  \end{align*}
\end{lemma}
So the right Haar measure is
\begin{align*}
  d\nu_r(E_k) = \prod_{i=2}^{p_k}E_{k[i,i]}^{-p_k + i - 1}
\end{align*}
\begin{proof}
  The proof is very similar to those in Propositions 5.13 and 5.14 of \cite{eaton1983multivariate}, noting that $T_{k[1,1]} = 1$.
\end{proof}

We'll eventually need the inverse transformation, which follows directly from Theorem 3 of chapter 8 section 4 of \cite{magnus1988matrix}.

\begin{lemma}
  \label{lemma:jacob.inv}
  For $E_k \in \mathcal{G}_{p_k}^1$, the Jacobian for the transformation $g(E_k) = E_k^{-1}$ is
  \begin{align}
    \label{equation:jacob.inv}
    \prod_{i = 2}^{p_k}E_{k[i,i]}^{-p_k - 1}
  \end{align}
\end{lemma}
\begin{proof}
  From \cite{magnus1988matrix}, $d(E_k^{-1}) = -E_k^{-1}(dE_k)E_k^{-1}$. Using Lemma \ref{lemma:jacob.left.right}, the Jacobian of the first transformation, $g_1(dE_k) = E_k^{-1}(dE_k)$  is $\prod_{i=2}^{p_k}E_{k[i,i]}^{-i}$. Jacobian of the second transformation $g_2(dE_k) = (dE_k)E_k^{-1}$ is $\prod_{i=2}^{p_k}E_{k[i,i]}^{-p_k + i - 1}$. Hence, overall Jacobian is (\ref{equation:jacob.inv}).
\end{proof}

Under this new parameterization, the likelihood is
\begin{align*}
  &p(X|\sigma,\Psi_1,\ldots,\Psi_K) \\
  &= (2\pi)^{np/2}|\sigma^2(\Psi_K\Psi_K^T \otimes \cdots \otimes \Psi_1\Psi_1^T)|^{-n/2}\\
  & \times \exp\{-||X \times \{\Psi_1^{-1},\ldots,\Psi_K^{-1},I_n\}||^T/(2\sigma^2)\}\}\\
  &\propto \sigma^{-np}\prod_{k=1}^{K}\prod_{i=2}^{p_k}\Psi_{k[i,i]}^{-np/p_k} \exp\{-||X \times \{\Psi_1^{-1},\ldots,\Psi_K^{-1},I_n\}||^T/(2\sigma^2)\}\},
\end{align*}
where $p = \prod_{k = 1}^Kp_k$. The (improper) prior is
\begin{align*}
  \pi(\sigma,\Psi_1,\ldots,\Psi_K) \propto \frac{1}{\sigma}\prod_{k = 1}^K\prod_{i=2}^{p_k}\Psi_{k[i,i]}^{i - p_k - 1}.
\end{align*}
Hence, the posterior is
\begin{align*}
  \sigma^{-np - 1}\prod_{k=1}^{K}\prod_{i=2}^{p_k}\Psi_{k[i,i]}^{i - np/p_k - p_k - 1} \exp\{-||X \times \{\Psi_1^{-1},\ldots,\Psi_K^{-1},I_n\}||^T/(2\sigma^2)\}.
\end{align*}
Since $\sigma^2|\Psi \sim \mbox{ inverse-gamma}(np/2,||X \times \{\Psi_1^{-1},\ldots,\Psi_K^{-1},I_n\}||^2/2$), we can integrate out $\sigma^2$, obtaining
\begin{align*}
  \pi(\Psi_1,\ldots,\Psi_K|X) \propto ||X \times \{\Psi_1^{-1},\ldots,\Psi_K^{-1},I_n\}||^{-np}\prod_{k=1}^{K}\prod_{i=2}^{p_k}\Psi_{k[i,i]}^{i - np/p_k - p_k - 1}.
\end{align*}
Let $S = X_{(K+1)}^TX_{(K+1)}$, the sample covariance matrix, then
\begin{align*}
  &\pi(\Psi_1,\ldots,\Psi_K|X)\\
  &\propto \tr[S(\Psi_K^{-T}\Psi_K^{-1}\otimes\cdots\otimes\Psi_1^{-T}\Psi_1^{-1})]^{-np/2}\prod_{k=1}^{K}\prod_{i=2}^{p_k}\Psi_{k[i,i]}^{i - np/p_k - p_k - 1}.
\end{align*}
Let $L_k = \Psi_k^{-1}$ for $k = 1,\ldots,K$. Then, using Lemma \ref{lemma:jacob.inv}, we have
\begin{align}
  \label{equation:post.l}
  \nonumber &\pi(L_1,\ldots,L_K|X)\\
  &\propto \tr[S(L_K^{T}L_K\otimes\cdots\otimes L_1^{T}L_1)]^{-np/2}\prod_{k=1}^{K}\prod_{i=2}^{p_k}L_{k[i,i]}^{np/p_k - i}
\end{align}
The posterior density is integrable if and only if (\ref{equation:post.l}) is integrable. We will now prove that when $n > \prod_{k=1}^K p_k$ then (\ref{equation:post.l}) is integrable. First, consider, consider the integral over $\mathcal{G}_{p}^1$, where $p = \prod_{k=1}^Kp_k$,
\begin{align}
\label{equation:v.s}
 &\int_{\mathcal{G}_{p}^1}\tr\left(VSV^T\right)^{-np/2}\prod_{i=2}^{p}V_{[i,i]}^{np - p - 1}dV
\end{align}
 Let $e = (1,0,\ldots,0)^T$, the vector of length $p$ with a $1$ in the first position and $0$'s everywhere else. Then $V = (e^T, V_2^T)^T$ and
\begin{align*}
  \tr&(VSV^T) = tr(e_1^TSe_1) + tr(V_2SV_2^T) = S_{[1,1]} + tr(V_2SV_2^T)\\
  &=(1 + tr(V_2SV_2^T)/S_{[1,1]})S_{[1,1]} =(1 + tr(V_2S_T(V_2S_T)^T)/S_{[1,1]})S_{[1,1]},
\end{align*}
where $S = S_TS_T^T$ is the lower triangular Cholesky decomposition of $S$. Let $W = V_2S_T$, so $V_2 = WS_T^{-1}$. The Jacobian of this transformation is $S_{T[1,1]}^{1-p}\prod_{i = 2}^pS_{T[i,i]}^{i-p-1}$ (same as the Jacobian in Proposition 5.14 of \cite{eaton1983multivariate} except we have one less $S_{T[1,1]}$ term). Then Equation (\ref{equation:v.s}) is proportional to
\begin{align*}
&\int_{\mathcal{G}_{p}^1}\left((1 + \tr(WW^T)/S_{[1,1]}\right)^{-np/2}\prod_{i=2}^{p}W_{[i,i]}^{np - p - 1}dW\\
&=\int_{\mathcal{G}_{p}^1}\left((1 + \mathbf{w}D\mathbf{w}/(np - p)\right)^{-(np - p + p)/2}\prod_{i=2}^{p}W_{[i,i]}^{np - p - 1}dW,
\end{align*}
where $\mathbf{w}$ is a vector containing all the non-zero elements of $W$ and $D = (n - p)I_p/S_{[1,1]}$. Notice that $\left((1 + \mathbf{w}D\mathbf{w}/(np - p)\right)^{-(np - p + p)/2}$ is the kernel of a multivariate $T$ distribution with  degrees of freedom $np - p$ and scale matrix $D^{-1} = S_{[1,1]}I_{p}/(np-p)$ \citep[equation (1.1)]{kotz2004multivariate}. Note that $E[W_{[i,j]}^{\nu}] < \infty$ if $\nu < n - p$ \citep[section 1.7]{kotz2004multivariate}. In particular, $n - p - 1 < n - p$. Hence
\begin{align*}
&\int_{\mathcal{G}_{p}^1}\tr\left(VSV^T\right)^{-np/2}\prod_{i=2}^{p}V_{[i,i]}^{np - p - 1}dV < \infty
\end{align*}
Using this, we have the following inequalities:
\begin{align*}
\infty &> \int_{\mathcal{G}_{p}^1} \tr\left[VSV^T\right]^{-np/2}\prod_{i=1}^pV_{k[i,i]}^{np - p - 1}dV\\
&= \int_{\mathcal{G}_{p}^1} \tr\left[VSV^T\right]^{-np/2}|V|^{np - p - 1}dV\\
&\geq \int_{\mathcal{G}_{p_1}^1 \times \cdots \times \mathcal{G}_{p_K}^1} \tr\left[(L_K\otimes\cdots\otimes L_1)S(L_K\otimes\cdots\otimes L_1)^T\right]^{-np/2}\\
&\times|L_K\otimes\cdots\otimes L_1|^{np - p- 1}dL_1\cdots dL_K\\
&=\int_{\mathcal{G}_{p_1}^1 \times \cdots \times \mathcal{G}_{p_K}^1} \tr\left[S(L_K^TL_K\otimes\cdots\otimes L_1^TL_1)\right]^{-np/2}\\
&\times\prod_{k=1}^K \prod_{i=2}^{p_k} L_{k[i,i]}^{(np - p - 1)p/p_k}dL_1\cdots dL_K,
\end{align*}
where the second inequality results from integrating over a smaller space. Note the following results: (1) $(np - p - 1)p/p_k \geq np/p_k - i_k$ for all $k = 1,\ldots,K$ and $i_k = 2,\ldots,p_k$ if $n \geq p$, (2) $L_{k[i,i]} > 0$, and (3) $E[|X|^{r_1}] < \infty$ and $r_1 > r_2$ $\Rightarrow E[|X|^{r_2}] < \infty$. Hence,
\begin{align*}
\infty&> \int_{\mathcal{G}_{p_1}^1 \times \cdots \times \mathcal{G}_{p_K}^1} \tr[S(L_K^{T}L_K\otimes\cdots\otimes L_1^{T}L_1)]^{-np/2}\\
&\times\prod_{k=1}^{K}\prod_{i=2}^{p_k}L_{k[i,i]}^{np/p_k - i}dL_1\cdots dL_K
\end{align*}
and the result is proved.

\subsection{Proof of Lemma \ref{lemma:e.sw}}
\begin{proof}
Let $VV^T$ be the lower triangular Cholesky decomposition of a Wishart$_p(\nu,I_{p})$-distributed random matrix. Recall from Bartlett's decomposition \citep{bartlett1933theory} that the elements of $V$ are independent with
\begin{align*}
  V_{[i,i]}^2 \sim \chi^2_{\nu - i + 1} \mbox{ and } V_{[i,j]} \sim N(0,1).
\end{align*}
Let $S = V^TV$. For $i \neq j$, we have
\begin{align*}
E\left[S_{[i,j]}\right] &= E\left[\sum_{k = 1}^p V_{[k,i]}V_{[k,j]}\right] = \sum_{k = 1}^pE\left[V_{[k,i]}\right] E\left[V_{[k,j]}\right].
\end{align*}
For $i \neq j$, we have either $E[V_{[k,i]}] = 0$ or $E[V_{[k,j]}] = 0$ for all $k = 1, \ldots, p$. Hence, $E[S_{[i,j]}] = 0$ for all $i \neq j$.

For $i = j$, we have
\begin{align*}
E\left[S_{[i,j]}\right] &= E\left[\sum_{k = 1}^p V_{[k,i]}V_{[k,j]}\right] = \sum_{k = 1}^pE\left[V_{[k,i]}^2\right] = E\left[V_{[i,i]}^2\right] + \sum_{k = i + 1}^{p}E\left[V_{[k,i]}^2\right]\\
&= \nu - i + 1 + \sum_{k = i + 1}^{p}1 = \nu - i + 1 + p - i = \nu + p + 1 - 2i.
\end{align*}
This expectation has been calculated in other papers \citep[for example]{james1961estimation,eaton1987best}.
\end{proof}

\subsection{Proof of Lemma \ref{lemma:det.to.lower}}
\begin{proof}
  We proceed by invariance arguments. The Jacobian, $J(\sigma,\Psi)$, is the unique continuous function that satisfies
  \begin{align*}
    \int_{G_{p_k}^+} f(L) \frac{dL}{\prod_{i=1}^{p_k} L_{[i,i]}^{p_k - i + 1}} &= \int_{\mathbb{R} \times \mathcal{G}_{p_k}^+} f(\sigma\Psi) \frac{J(\sigma,\Psi) d\sigma d\Psi}{\prod_{i=1}^{p_k} (\sigma \Psi_{[i,i]})^{p_k - i + 1}}\\
    &= \int_{\mathbb{R} \times \mathcal{G}_{p_k}^+} f(\sigma\Psi) \frac{J(\sigma,\Psi) d\sigma d\Psi}{\sigma^{p_k(p_k+1)/2}\prod_{i=1}^{p_k} \Psi_{[i,i]}^{p_k-i+1}},
  \end{align*}
  where $dL/(\prod_{i=1}^{p_k} L_{[i,i]}^{p_k-i+1})$ is a right invariant measure with respect to the action $L \mapsto LA$ on $G_{p_k}^+$ for $A \in G_{p_k}^+$ \citep[Proposition 5.14]{eaton1983multivariate}. Hence, this invariance property must also hold for the right integral. So for $b > 0$ and $B \in \mathcal{G}_{p_k}^+$, we have that $bB \in G_{p_k}^+$ and
  \begin{align*}
    &\int_{\mathbb{R} \times \mathcal{G}_{p_k}^+} f(\sigma\Psi) \frac{J(\sigma,\Psi) d\sigma d\Psi}{\sigma^{p_k(p_k+1)/2}\prod_{i=1}^{p_k} \Psi_{[i,i]}^{p_k - i + 1}}\\
    & = \int_{\mathbb{R} \times \mathcal{G}_{p_k}^+} f(b\sigma \Psi B) \frac{J(\sigma,\Psi) d\sigma d\Psi}{\sigma^{p_k(p_k+1)/2}\prod_{i=1}^{p_k} \Psi_{[i,i]}^{p_k - i +1}}.
  \end{align*}
  So making the change of variables $\sigma = e/b$ and $\Psi = EB^{-1}$ , we have
  \begin{align*}
    &\int_{\mathbb{R} \times \mathcal{G}_{p_k}^+} f(b\sigma\Psi B) \frac{J(\sigma,\Psi) d\sigma d\Psi}{\sigma^{p_k(p_k+1)/2}\prod_{i=1}^{p_k} \Psi_{[i,i]}^{p_k - i + 1}}\\
    &= \int_{\mathbb{R} \times \mathcal{G}_{p_k}^+} f(eE) \frac{\frac{1}{b} \prod_{i=2}^{p_k}B_{[i,i]}^{i - 2} J(e/b,EB^{-1}) de dE}{(e/b)^{p_k(p_k+1)/2}\prod_{i=1}^{p_k} E_{[i,i]}^{p_k - i + 1} B_{[i,i]}^{i - p_k - 1}}\\
    &= \int_{\mathbb{R} \times \mathcal{G}_{p_k}^+} f(eE) \frac{b^{p_k(p_k+1)/2 - 1} B_{[1,1]}^{p_k}\prod_{i=2}^{p_k}B_{[i,i]}^{p_k-1} J(e/b,EB^{-1}) de dE}{e^{p_k(p_k+1)/2}\prod_{i=1}^{p_k} E_{[i,i]}^{p_k - i + 1}}\\
    &= \int_{\mathbb{R} \times \mathcal{G}_{p_k}^+} f(eE) \frac{b^{p_k(p_k+1)/2 - 1} B_{[1,1]} J(e/b,EB^{-1}) de dE}{e^{p_k(p_k+1)/2}\prod_{i=1}^{p_k} E_{[i,i]}^{p_k - i + 1}},
  \end{align*}
  where we used (\ref{equation:gpk.jacobian}) for the first equality and our parameterization of $\mathcal{G}_{p_k}^+$, $\prod_{i = 2}^{p_k}B_{[i,i]}^{-1} = B_{[1,1]}$, for the last equality. So we must have that
  \begin{align*}
    J(\sigma,\Psi) = b^{p_k(p_k+1)/2 - 1} B_{[1,1]} J(\sigma/b,\Psi B^{-1}).
  \end{align*}
  Set $B = \Psi$ and $b = \sigma$ to obtain: $J(\sigma,\Psi) = \sigma^{p_k(p_k+1)/2 - 1} \Psi_{[1,1]} J(1,I)$, where $J(1,I)$ is a constant.
\end{proof}

\subsection{Proof of Lemma \ref{lemma:bartlett.inv.wishart}}
Let $S^{-1} \sim \mbox{Wishart}_p(\nu,I_{p})$ and partition $S^{-1}$ and $S \sim \mbox{inverse-Wishart}_p(\nu,I_{p})$ conformably such that $p_1 + p_2 = p$:
\begin{align*}
  S^{-1} = 
  \left(
    \begin{array}{cc}
      S^{11} & S^{12}\\
      S^{21} & S^{22}\\
    \end{array}
  \right),
  \hspace{5 mm} S =
  \left(
    \begin{array}{cc}
      S_{11} & S_{12}\\
      S_{21} & S_{22}\\
    \end{array}
  \right).
\end{align*}
Denote $S^{11 \bullet 2} = S^{11} - S^{12}(S^{22})^{-1}S^{21}$, the Schur complement. The following are well known properties of the Wishart distribution (see, for example, Proposition 8.7 of \cite{eaton1983multivariate})
\begin{align*}
S^{22} &\sim \mbox{Wishart}_{p_2}(I_{p_2},\nu)\\
S^{21}|S^{22} &\sim N_{p_2 \times p_1}(0,S^{22} \otimes I_{p_1})\\
S^{11\bullet 2} &\sim \mbox{Wishart}_{p_1}(I_{p_1},\nu - p_2)\\
S^{11\bullet 2} &\mbox{ is independent of } \{S^{22},S^{21}\}
\end{align*}
The relationship of the inverse of a partitioned matrix (see, for example, Section 0.7.3 of \cite{horn2012matrix}) implies that
\begin{align}
S_{11} &= (S^{11 \bullet 2})^{-1} \sim \mbox{inverse-Wishart}_{p_1}(I_{p_1},\nu - p_2) \label{equation:s11}\\
S_{22\bullet 1} &= (S^{22})^{-1} \sim \mbox{inverse-Wishart}_{p_2}(I_{p_2},\nu) \label{equation:s22dot1}\\
\begin{split}\label{equation:s21}
S_{21}|S_{11},S_{22\bullet 1} &\overset{d}{=} -(S^{22})^{-1}S^{21}(S^{11\bullet 2})^{-1}\\
&\sim N_{p_2 \times p_1}(0,(S^{22})^{-1} \otimes (S^{11\bullet 2})^{-1}(S^{11\bullet 2})^{-1})\\
&= N_{p_2 \times p_1}(0,S_{22\bullet 1} \otimes S_{11}S_{11}).
\end{split}
\end{align}
It is also well known that 
\begin{align}
\mbox{if } p = 1 \mbox{ then } S \sim \mbox{ inverse-gamma}(\nu/2,1/2). \label{equation:pequals1}
\end{align}

We should be able to use these results to come up with the distribution of the elements of the lower triangular Cholesky decomposition from an inverse-Wishart distributed random matrix, which seems surprisingly difficult to find in the literature.

\begin{proof}[Proof of Lemma \ref{lemma:bartlett.inv.wishart}]
  We proceed by induction on the dimension. It is clearly true for $n = 1$. Assume it is true for $n - 1$. Then partition $S_{[1:n,1:n]} \sim \mbox{ inverse-Wishart}_n(I_{n},\nu - p + n)$ such that the top left submatrix, $S_{11}$, is $n-1$ by $n-1$.
  \begin{align*}
    S_{[1:n,1:n]} = 
    \left(
      \begin{array}{cc}
        S_{11} & S_{12}\\
        S_{21} & s_{22}\\
      \end{array}
    \right)
    =
    \left(
      \begin{array}{cc}
        W_1 & 0 \\
        S_{21}W_1^{-T} & s_{22 \bullet 1}^{1/2} \\
      \end{array}
    \right)
    \left(
      \begin{array}{cc}
        W_1^T & W_1^{-1}S_{12} \\
        0 & s_{22 \bullet 1}^{1/2} \\
      \end{array}
    \right).
  \end{align*}

  Note that $S_{11} = W_1W_1^T$. Using (\ref{equation:s11})-(\ref{equation:pequals1}), we have that:
  \begin{align*}
    W_{[n,n]}^2 = s_{22 \bullet 1} &\sim \mbox{ inverse-gamma}\left( (\nu - p + n)/2, 1/2 \right)\\
    S_{21}W_1^{-T}|W_1,s_{22 \bullet 1} &= S_{21}{S_{11}^{-1/2}}^T | S_{11},s_{22\bullet 1}\\
& \sim N_{1 \times n - 1}\left(0,\left(s_{22 \bullet 1} \otimes W_1^TW_1\right)\right) \\
&=  N_{n-1}(0,s_{22 \bullet 1}W_1^TW_1) \\
&= N_{n-1}(0,W_{[n,n]}^2W_1^TW_1).
  \end{align*}
\end{proof}

\subsection{Proof of Lemma \ref{lemma:lt.wishart.inverse}}

\begin{proof}
  We proceed by induction on the dimension. It is clearly true for $n = 1$. Assume it is true for $n-1$. Note that for lower triangular matrices, the $[1:n,1:n]$ submatrix of the inverse is the inverse of the $[1:n,1:n]$ submatrix. Hence, partition $W_{k[1:n,1:n]} = V_{k[1:n,1:n]}^{-1}$  by:
  \begin{align*}
    V_{k[1:n,1:n]} = 
    \left(
      \begin{array}{cc}
        V_{11} & 0 \\
        V_{21} & v_{22} \\
      \end{array}
    \right),
    \hspace{5 mm}
    W_{k[1:n,1:n]} = 
    \left(
      \begin{array}{cc}
        W_{11} & 0 \\
        W_{21} & w_{22} \\
      \end{array}
    \right),
  \end{align*}
  where the top left submatrix is $n-1$ by $n-1$. Then $v_{22}^2 = 1/w_{22}^2$ is clearly $\chi^2_{\nu - n + 1}$. We have that $V_{21} = -w_{22}^{-1}W_{21}W_{11 \bullet 2}^{-1}$. Also, $W_{11 \bullet 2} = W_{11} - W_{21}*0 / w_{22} = W_{11}$.
  Since
  \begin{align*}
    W_{21}|W_{11},w_{22} \sim N_{n-1}\left(0,w_{22}^2 W_{11}^TW_{11}\right),
  \end{align*}
  we have that
  \begin{align*}
    -w_{22}^{-1}W_{21}W_{11}^{-1}|W_{11},w_{22} \sim N_{n-1}\left(0,I\right).
  \end{align*}
  Hence, the result is proved.
\end{proof}

\subsection{Proof of Proposition \ref{prop:mree}}
\begin{proof}
  This minimization problem is equivalent to minimizing
  \begin{align*}
    &s^2 \sum_{k=1}^{K} \frac{p}{p_k} \tr\left( S_k E\left[\left(\sigma^2\Sigma_k\right)^{-1}\right]\right) - Kp\log\left( s^2 \right)\\
    &= s^2 \sum_{k=1}^{K} \frac{p}{p_k} \tr\left( S_k \mathcal{E}_k^{-1}\right) - Kp\log\left( s^2 \right).
  \end{align*}
  
  Let us absorb the scale parameter into $S_k$. That is, let $\tilde{S}_k = s^2 S_k$, then $s^2 = |\tilde{S}_k|^{1/p_k}$, and we wish to minimize with respect to $\tilde{S}_k$:
  \begin{align*}
    \frac{p}{p_k} \tr\left(\tilde{S}_k \mathcal{E}_k^{-1} \right) - \frac{Kp}{p_k}\log\left( |\tilde{S}_k| \right) + |\tilde{S}_k|^{1/p_k} \sum_{j \neq k} \frac{p}{p_j} \tr\left( S_j^T \mathcal{E}_j^{-1}\right).
  \end{align*}
  Letting $\lambda = \frac{p_k}{p}\sum_{j \neq k} \frac{p}{p_j} \tr\left( S_j^T \mathcal{E}_j^{-1}\right)$, this is equivalent to minimizing:
  \begin{align*}
    \tr\left(\tilde{S}_k \mathcal{E}_k^{-1} \right) - K\log\left( |\tilde{S}_k| \right) + |\tilde{S}_k|^{1/p_k} \lambda
  \end{align*}
  with respect to $\tilde{S}_k$.\\
  
  Since the mapping $\tilde{S}_k \mapsto \mathcal{E}_k^{-1/2} \tilde{S}_k \mathcal{E}_k^{-1/2} = \Omega$ is a bijection of the set of $p_k \times p_k$ symmetric positive definite matrices, we can write:
  \begin{align*}
    &\min_{\tilde{S}_k > 0} \left\{ \tr\left(\tilde{S}_k \mathcal{E}_k^{-1} \right) - K\log\left( |\tilde{S}_k| \right) + |\tilde{S}_k|^{1/p_k} \lambda \right\}\\
    &= \min_{\Omega > 0} \left\{ \tr\left(\Omega \right) - K\log\left( |\Omega| \right) + |\Omega|^{1/p_k} \lambda^* + K\log(|\mathcal{E}_k|) \right\}\\
    &= \min_{\omega_1 \geq \cdots \geq \omega_{p_k} > 0} \left\{ \sum_{i=1}^{p_k} \omega_i - K\sum_{i = 1}^{p_k} \log(\omega_i) + \lambda^* \prod_{i=1}^{p_k}\omega_i^{1/p_k} \right\},
  \end{align*}
  where $\lambda^* = \lambda|\mathcal{E}_k|^{1/p_k}$ and $\omega_1,\omega_2,\ldots,\omega_{p_k}$ are the ordered eigenvalues of $\Omega$. Taking derivatives with respect to $\omega_j$ and setting equal to 0, we have:
  \begin{align*}
    &1 - \frac{K}{\omega_j} + \frac{1}{p_k}\omega_j ^{1/p_k -1}\lambda^* \prod_{i \neq j} \omega_i^{1/p_k} = 0\\
    &\Leftrightarrow \omega_j =  K - \frac{1}{p_k}\lambda^* \prod_{i = 1}^{p_k} \omega_i^{1/p_k} \mbox{ for all } j = 1,\ldots,p_k.
  \end{align*}
  So all of the eigenvalues have the same critical value.\\
  
  Taking second derivatives, we have:
  \begin{align*}
    &\frac{K}{\omega_j^2} - \frac{p_k - 1}{p_k^2}\lambda^* \omega^{1/p_k - 2} \prod_{i \neq j}^{p_k} \omega_j^{1/p_k} > 0 \Leftrightarrow K - \frac{p_k - 1}{p_k^2}\lambda^* \prod_{j = 1}^{p_k} \omega_j^{1/p_k} > 0\\
    \Leftrightarrow &K + \frac{p_k - 1}{p_k}\left(K - \frac{1}{p_k}\lambda^* \prod_{j = 1}^{p_k} \omega_j^{1/p_k} - K \right) > 0\\
    \Leftrightarrow &K + \frac{p_k - 1}{p_k}\left(\omega_j - K \right) > 0 \Leftrightarrow \frac{p_k - 1}{p_k}\omega_j + K\frac{1}{p_k} > 0.
  \end{align*}
  Hence, by a second derivative test, this critical value is a minimizer for all $\omega_j$. This is a global minimum since
  \[
  \mbox{as } \omega_1 \rightarrow \infty \mbox{ we have that } \left\{ \sum_{i=1}^{p_k} \omega_i - K\sum_{i = 1}^{p_k} \log(\omega_i) + \lambda^* \prod_{i=1}^{p_k}\omega_i^{1/p_k} \right\} \rightarrow \infty
  \]
  and
  \[
  \mbox{as } \omega_{p_k} \rightarrow 0 \mbox{ we have that } \left\{ \sum_{i=1}^{p_k} \omega_i - K\sum_{i = 1}^{p_k} \log(\omega_i) + \lambda^* \prod_{i=1}^{p_k}\omega_i^{1/p_k} \right\} \rightarrow \infty.
  \]
  This implies that all of the $\omega_j$ are equal. In particular, that $\omega_j = (Kp_k)/(p_k + \lambda^*)$ for all $j = 1,\ldots,p_k$. This in turn implies that $\Omega$ is a constant multiple of the identity. Thus, the $\tilde{S}_k$ that minimizes the risk given all $S_j$ such that $j \neq k$ is:
  \begin{align*}
    \tilde{S}_k = \frac{Kp_k}{p_k + \lambda^*} \mathcal{E}_k.
  \end{align*}
  But this means that the $S_k$ that minimizes this risk, no matter what the other $S_j$'s are, is
  \begin{align*}
    \hat{\Sigma}_k = \mathcal{E}_k / |\mathcal{E}_k|^{1/(p_k)}.
  \end{align*}  
  It remains to minimize with respect to $s$. The minimizer is the $s$ such that
  \begin{align*}
    2s\sum_{k=1}^{K} \frac{p}{p_k} \tr\left( \hat{\Sigma}_k \mathcal{E}_k^{-1}\right) - \frac{2Kp}{s} = 0.
  \end{align*}
  And solving for $s$ we get
  \begin{align*}
    \hat{\sigma}^2 = \frac{K}{\sum_{k=1}^{K} \frac{1}{p_k}\tr\left( \hat{\Sigma}_k \mathcal{E}_k^{-1} \right) }.
  \end{align*}
  But since $\hat{\Sigma}_k = \mathcal{E}_k / |\mathcal{E}_k|^{1/(p_k)}$, we have that
  \begin{align*}
    \hat{\sigma}^2 = \frac{K}{\sum_{k=1}^{K} |\mathcal{E}_k|^{-1/p_k}}.    
  \end{align*}
\end{proof}

\subsection{Proof of Proposition \ref{prop:takemura}}
\begin{proof}
Let $\Phi_k = \Sigma_k/\tr(\Sigma_k)$, $D_k = S_k / \tr(S_k)$ for $k = 1,\ldots,K$. So $\Sigma_k = \Phi_k/|\Phi_k|^{1/p_k}$ and $S_k = D_k/|D_k|^{1/p_k}$ for $k = 1,\ldots,K$. $\Phi_k$ and $D_k$ both have trace 1. The space of trace 1 symmetric positive definite matrices is convex.  Let $\Phi = (\sigma^2,\Phi_1,\ldots,\Phi_K)$ and $D = (s^2,D_1,\ldots,D_K)$. Define
\begin{align*}
  &L_2\left(\Phi, D\right) = \frac{s^2}{\sigma^2} \sum_{k=1}^{K}\frac{p}{p_k}|D_k\Phi_k^{-1}|^{-1/p_k}\tr\left(D_k\Phi_k^{-1}\right) - Kp\log\left(\frac{s^2}{\sigma^2}\right) - Kp.\\
&\text{So, } L_M\left(\Sigma,S\right) = L_2\left( \Phi, D\right).\\
&\text{Hence, } E\left[\left. L_M\left(\Sigma,S\right)\right| X\right] = E\left[\left. L_2\left( \Phi, D\right)\right| X \right].
\end{align*}
So if $L_2$ is convex in each $D_k$, we can uniformly decrease the risk. That is, given $B_k,E_k \in \mathcal{G}_{p_k}^{+}$ are two estimators from two different special linear group transformations, an estimator that uniformly decreases the risk is found by setting $F_k = (B_k/\tr(B_k) + E_k/\tr(E_k))/2$ and using $F_k/|F_k|^{1/p_k}$ as our estimator. Averaging over the whole space of orthogonal matrices will result in an orthogonally equivariant estimator.

It remains to prove that $L_2$ is convex in each $D_k$. It suffices to show that $|D_k|^{-1/p_k}\tr(D_k\Phi_k^{-1})$ is convex in $D_k$. Since, for $\alpha \in [0,1]$, $\tr((\alpha D_k + (1-\alpha)E_k)\Phi_k^{-1}) = \alpha\tr(D_k\Phi_k^{-1}) + (1-\alpha)\tr(E_k\Phi_k^{-1})$ is convex in $D_k$, if $|D_k|^{-1/p_k}$ is also convex, then we are done. $|D_k|$ is a concave function \citep[Theorem 1]{cover1988determinant}, and $f(x) = \log(x)$ is concave monotonic, so $\log(|D_k|)$ is concave, so $-\log(|D_k|)/p_k$ is convex, so $\exp(-\log(|D_k|)/p_k) = |D_k|^{-1/p_k}$ is convex.

 We also have that $c b^2 - h \log(b^2)$ is convex in $b^2$ for $c,h > 0$, so we can average the scale estimates to decrease risk as well.

To summarize, we have:
\begin{align*}
L_M&\left(\Sigma,\left(f^2,F_1/|F_1|^{1/p_1},\ldots,F_K/|F_K|^{1/p_K}\right)\right)\\
=&L_2\left(\Phi,\left(f^2,F_1,\ldots, F_K\right)\right)\\
=&L_2\left(\Phi,\left((b^2+e^2)/2,B_1/\tr\left(B_1\right) + E_1/\tr\left(E_1\right)\right)/2\right. ,\\
&\left.\left.\ldots,\left(B_K/\tr\left(B_K\right) + E_K/\tr\left(E_K\right)\right)/2\right)\right)\\
\leq& \frac{1}{2}L_2\left(\Phi,\left(b^2,B_1/\tr\left(B_1\right),\ldots, B_K/\tr\left(B_K\right)\right)\right)\\
& + \frac{1}{2}L_2\left(\Phi,\left(e^2,E_1/\tr\left(E_1\right),\ldots,E_K/\tr\left(E_K\right)\right)\right)\\
=& \frac{1}{2}L_M\left(\Sigma,B\right) + \frac{1}{2}L_M\left(\Sigma,E\right).
\end{align*}
If $B$ and $E$ have the same (constant) risk as the UMREE, $\hat{\Sigma}(X)$, then
\begin{align*}
&E\left[L_M\left(\Sigma,\left(f^2,F_1/|F_1|^{1/p_1},\ldots,F_K/|F_K|^{1/p_K}\right)\right)\right]\\
&\leq \frac{1}{2}E\left[L_M\left(\Sigma,B\right)\right] + \frac{1}{2}\left[L_M\left(\Sigma,E\right)\right]\\
&= E\left[L_M\left(\Sigma,\hat{\Sigma}(X)\right)\right]
\end{align*}
\end{proof}

\bibliography{lt_bib}

\begin{thebibliography}{40}
\providecommand{\natexlab}[1]{#1}
\providecommand{\url}[1]{\texttt{#1}}
\expandafter\ifx\csname urlstyle\endcsname\relax
  \providecommand{\doi}[1]{doi: #1}\else
  \providecommand{\doi}{doi: \begingroup \urlstyle{rm}\Url}\fi

\bibitem[Akdemir and Gupta(2011)]{akdemir2011array}
Deniz Akdemir and Arjun~K Gupta.
\newblock Array variate random variables with multiway {K}ronecker delta
  covariance matrix structure.
\newblock \emph{Journal of algebraic statistics}, 2\penalty0 (1), 2011.

\bibitem[Bartlett(1933)]{bartlett1933theory}
MS~Bartlett.
\newblock On the theory of statistical regression.
\newblock \emph{Proceedings of the Royal Society of Edinburgh}, 53:\penalty0
  260--283, 1933.

\bibitem[Bondar and Milnes(1981)]{bondar1981amenability}
James~V Bondar and Paul Milnes.
\newblock Amenability: {A} survey for statistical applications of
  {H}unt-{S}tein and related conditions on groups.
\newblock \emph{Zeitschrift f{\"u}r Wahrscheinlichkeitstheorie und verwandte
  Gebiete}, 57\penalty0 (1):\penalty0 103--128, 1981.

\bibitem[Bro(2006)]{bro2006review}
Rasmus Bro.
\newblock Review on multiway analysis in chemistry - 2000--2005.
\newblock \emph{Critical reviews in analytical chemistry}, 36\penalty0
  (3-4):\penalty0 279--293, 2006.

\bibitem[Cichocki et~al.(2014)Cichocki, Mandic, Caiafa, Phan, Zhou, Zhao, and
  De~Lathauwer]{cichockitensor}
A~Cichocki, D~Mandic, C~Caiafa, AH~Phan, G~Zhou, Q~Zhao, and L~De~Lathauwer.
\newblock Tensor decompositions for signal processing applications.
\newblock \emph{From Two-way to Multiway Component Analysis, ESAT-STADIUS
  Internal Report}, pages 13--235, 2014.

\bibitem[Cover and Thomas(1988)]{cover1988determinant}
Thomas~M Cover and A~Thomas.
\newblock Determinant inequalities via information theory.
\newblock \emph{SIAM journal on Matrix Analysis and Applications}, 9\penalty0
  (3):\penalty0 384--392, 1988.

\bibitem[Dawid(1981)]{dawid1981some}
A~Philip Dawid.
\newblock Some matrix-variate distribution theory: notational considerations
  and a {B}ayesian application.
\newblock \emph{Biometrika}, 68\penalty0 (1):\penalty0 265--274, 1981.

\bibitem[De~Lathauwer et~al.(2000{\natexlab{a}})De~Lathauwer, De~Moor, and
  Vandewalle]{de2000best}
Lieven De~Lathauwer, Bart De~Moor, and Joos Vandewalle.
\newblock On the best rank-1 and rank-(${R}_1,{R}_2,\ldots,{R}_n$)
  approximation of higher-order tensors.
\newblock \emph{SIAM Journal on Matrix Analysis and Applications}, 21\penalty0
  (4):\penalty0 1324--1342, 2000{\natexlab{a}}.

\bibitem[De~Lathauwer et~al.(2000{\natexlab{b}})De~Lathauwer, De~Moor, and
  Vandewalle]{de2000multilinear}
Lieven De~Lathauwer, Bart De~Moor, and Joos Vandewalle.
\newblock A multilinear singular value decomposition.
\newblock \emph{SIAM journal on Matrix Analysis and Applications}, 21\penalty0
  (4):\penalty0 1253--1278, 2000{\natexlab{b}}.

\bibitem[De~Silva and Lim(2008)]{de2008tensor}
Vin De~Silva and Lek-Heng Lim.
\newblock Tensor rank and the ill-posedness of the best low-rank approximation
  problem.
\newblock \emph{SIAM Journal on Matrix Analysis and Applications}, 30\penalty0
  (3):\penalty0 1084--1127, 2008.

\bibitem[Dutilleul(1999)]{dutilleul1999mle}
Pierre Dutilleul.
\newblock The {MLE} algorithm for the matrix normal distribution.
\newblock \emph{Journal of statistical computation and simulation}, 64\penalty0
  (2):\penalty0 105--123, 1999.

\bibitem[Eaton(1983)]{eaton1983multivariate}
Morris~L Eaton.
\newblock \emph{Multivariate statistics: a vector space approach}.
\newblock Wiley New York, 1983.

\bibitem[Eaton(1989)]{eaton1989group}
Morris~L Eaton.
\newblock Group invariance applications in statistics.
\newblock In \emph{Regional conference series in Probability and Statistics},
  pages i--133. JSTOR, 1989.

\bibitem[Eaton et~al.(1987)Eaton, Olkin, et~al.]{eaton1987best}
Morris~L Eaton, Ingram Olkin, et~al.
\newblock Best equivariant estimators of a cholesky decomposition.
\newblock \emph{The Annals of Statistics}, 15\penalty0 (4):\penalty0
  1639--1650, 1987.

\bibitem[Fuentes(2006)]{fuentes2006testing}
Montserrat Fuentes.
\newblock Testing for separability of spatial--temporal covariance functions.
\newblock \emph{Journal of statistical planning and inference}, 136\penalty0
  (2):\penalty0 447--466, 2006.

\bibitem[Haff(1991)]{haff1991variational}
LR~Haff.
\newblock The variational form of certain {B}ayes estimators.
\newblock \emph{The Annals of Statistics}, pages 1163--1190, 1991.

\bibitem[Hoff(2011)]{hoff2011separable}
Peter~D Hoff.
\newblock Separable covariance arrays via the {T}ucker product, with
  applications to multivariate relational data.
\newblock \emph{Bayesian Analysis}, 6\penalty0 (2):\penalty0 179--196, 2011.

\bibitem[Horn and Johnson(2012)]{horn2012matrix}
Roger~A Horn and Charles~R Johnson.
\newblock \emph{Matrix analysis}.
\newblock Cambridge university press, 2012.

\bibitem[James and Stein(1961)]{james1961estimation}
William James and Charles Stein.
\newblock Estimation with quadratic loss.
\newblock In \emph{Proceedings of the fourth Berkeley symposium on mathematical
  statistics and probability}, volume~1, pages 361--379, 1961.

\bibitem[Kiefer(1957)]{kiefer1957invariance}
J~Kiefer.
\newblock Invariance, minimax sequential estimation, and continuous time
  processes.
\newblock \emph{The Annals of Mathematical Statistics}, 28\penalty0
  (3):\penalty0 573--601, 1957.

\bibitem[Kiers and Mechelen(2001)]{kiers2001three}
Henk~AL Kiers and Iven~Van Mechelen.
\newblock Three-way component analysis: Principles and illustrative
  application.
\newblock \emph{Psychological methods}, 6\penalty0 (1):\penalty0 84, 2001.

\bibitem[Kolda and Bader(2009)]{kolda2009tensor}
Tamara~G Kolda and Brett~W Bader.
\newblock Tensor decompositions and applications.
\newblock \emph{SIAM review}, 51\penalty0 (3):\penalty0 455--500, 2009.

\bibitem[Kotz and Nadarajah(2004)]{kotz2004multivariate}
Samuel Kotz and Saralees Nadarajah.
\newblock \emph{Multivariate t-distributions and their applications}.
\newblock Cambridge University Press, 2004.

\bibitem[Kroonenberg(2008)]{kroonenberg2008applied}
Pieter~M Kroonenberg.
\newblock \emph{Applied multiway data analysis}, volume 702.
\newblock John Wiley \& Sons, 2008.

\bibitem[Lin and Perlman(1985)]{lin1985monte}
SP~Lin and Michael~D Perlman.
\newblock A {M}onte {C}arlo comparison of four estimators of a covariance
  matrix.
\newblock \emph{Multivariate Analysis}, 6:\penalty0 411--429, 1985.

\bibitem[Magnus and Neudecker(1988)]{magnus1988matrix}
X~Magnus and Heinz Neudecker.
\newblock Matrix differential calculus.
\newblock \emph{New York}, 1988.

\bibitem[Mardia(1993)]{mardia1993spatial}
Kanti~V. Mardia.
\newblock Spatial-temporal analysis of multivariate environmental monitoring
  data.
\newblock In Ganapati~P Patil and Calyampudi~Radhakrishna Rao, editors,
  \emph{Multivariate environmental statistics"}, volume~6. Elsevier, 1993.

\bibitem[Ohlson et~al.(2013)Ohlson, Rauf~Ahmad, and
  Von~Rosen]{ohlson2013multilinear}
Martin Ohlson, M~Rauf~Ahmad, and Dietrich Von~Rosen.
\newblock The multilinear normal distribution: {I}ntroduction and some basic
  properties.
\newblock \emph{Journal of Multivariate Analysis}, 113:\penalty0 37--47, 2013.

\bibitem[Rotman(1995)]{rotman1995introduction}
Joseph~J Rotman.
\newblock \emph{An introduction to the theory of groups}, volume 148.
\newblock Springer, 1995.

\bibitem[Shitan and Brockwell(1995)]{shitan1995asymptotic}
Mahendran Shitan and Peter~J Brockwell.
\newblock An asymptotic test for separability of a spatial autoregressive
  model.
\newblock \emph{Communications in Statistics-Theory and Methods}, 24\penalty0
  (8):\penalty0 2027--2040, 1995.

\bibitem[Smilde et~al.(2005)Smilde, Bro, and Geladi]{smilde2005multi}
Age Smilde, Rasmus Bro, and Paul Geladi.
\newblock \emph{Multi-way analysis: applications in the chemical sciences}.
\newblock John Wiley \& Sons, 2005.

\bibitem[Srivastava and Khatri(1979)]{srivastava1979introduction}
MS~Srivastava and CG~Khatri.
\newblock An introduction to multivariate statistics., 1979.

\bibitem[Stein(1975)]{stein1975estimation}
Charles Stein.
\newblock Estimation of a covariance matrix.
\newblock \emph{Rietz Lecture}, 1975.

\bibitem[Takemura(1983)]{takemura1983orthogonally}
Akimichi Takemura.
\newblock An orthogonally invariant minimax estimator of the covariance matrix
  of a multivariate normal population.
\newblock Technical report, DTIC Document, 1983.

\bibitem[Tao et~al.(2005)Tao, Li, Hu, Maybank, and Wu]{tao2005supervised}
Dacheng Tao, Xuelong Li, Weiming Hu, Stephen Maybank, and Xindong Wu.
\newblock Supervised tensor learning.
\newblock In \emph{Data Mining, Fifth IEEE International Conference on}, pages
  8--pp. IEEE, 2005.

\bibitem[Vasilescu and Terzopoulos(2003)]{vasilescu2003multilinear}
M~Alex~O Vasilescu and Demetri Terzopoulos.
\newblock Multilinear subspace analysis of image ensembles.
\newblock In \emph{Computer Vision and Pattern Recognition, 2003. Proceedings.
  2003 IEEE Computer Society Conference on}, volume~2, pages II--93. IEEE,
  2003.

\bibitem[Wiesel(2012{\natexlab{a}})]{wiesel2012convexity}
Ami Wiesel.
\newblock On the convexity in {K}ronecker structured covariance estimation.
\newblock In \emph{Statistical Signal Processing Workshop (SSP), 2012 IEEE},
  pages 880--883. IEEE, 2012{\natexlab{a}}.

\bibitem[Wiesel(2012{\natexlab{b}})]{wiesel2012geodesic}
Ami Wiesel.
\newblock Geodesic convexity and covariance estimation.
\newblock \emph{Signal Processing, IEEE Transactions on}, 60\penalty0
  (12):\penalty0 6182--6189, 2012{\natexlab{b}}.

\bibitem[Yang and Berger(1994)]{yang_berger_1994}
Ruo-yong Yang and James~O. Berger.
\newblock Estimation of a covariance matrix using the reference prior.
\newblock \emph{Ann. Statist.}, 22\penalty0 (3):\penalty0 1195--1211, 1994.
\newblock ISSN 0090-5364.
\newblock \doi{10.1214/aos/1176325625}.
\newblock URL \url{http://dx.doi.org/10.1214/aos/1176325625}.

\bibitem[Zidek(1969)]{zidek1969representation}
James~V Zidek.
\newblock A representation of {B}ayes invariant procedures in terms of {H}aar
  measure.
\newblock \emph{Annals of the Institute of Statistical Mathematics},
  21\penalty0 (1):\penalty0 291--308, 1969.

\end{thebibliography}

\end{document}